\documentclass[11pt,twoside]{amsart}
\usepackage{t1enc}
\usepackage[latin2]{inputenc}

\usepackage{phaistos}

\usepackage{amsmath, amsthm, amscd, amssymb, txfonts, verbatim}
\usepackage[T1]{fontenc}
\usepackage{graphicx, nicefrac, url, color}
\usepackage{enumitem}
\usepackage{hyperref} 
\usepackage[normalem]{ulem}
\usepackage{cancel}
\usepackage{multicol}
\usepackage{wasysym}

\usepackage{mathpazo} 

\newcommand{\mc}{\mathcal}
\newcommand{\mbb}{\mathbb}
\newcommand{\mrm}{\mathrm}

\newcommand{\mbf}{\mathbf}

\newcommand{\0}{\emptyset}
\newcommand{\al}{\alpha}
\newcommand{\be}{\beta}
\newcommand{\ga}{\gamma}

\newcommand{\eps}{\varepsilon}

\newcommand{\om}{\omega}

\newcommand{\PP}{\mathbb{P}}
\newcommand{\NN}{\mathbb{N}}

\newcommand{\clrest}{\!\upharpoonright\!}
\newcommand{\wt}{\widetilde}
\newcommand{\dom}{\mathrm{dom}}

\newcommand{\bc}{\begin{center}}
\newcommand{\ec}{\end{center}}

\newcommand{\N}{\mathbb{N}}
\newcommand{\R}{\mathbb{R}}
\newcommand{\Q}{\mathbb{Q}}

\newcommand{\dotleq}{\,\dot{\leqslant}\,}

\DeclareMathOperator\supp{supp}

\DeclareMathOperator{\spa}{span}

\newtheorem*{nnclaim}{Claim}
\newtheorem*{nnthm}{Theorem}

\newtheorem{thm}{Theorem}[section]
\newtheorem{lem}[thm]{Lemma}
\newtheorem{prop}[thm]{Proposition}
\newtheorem{cor}[thm]{Corollary}
\newtheorem{fact}[thm]{Fact}

\theoremstyle{definition}
\newtheorem{que}[thm]{Question}

\newtheorem{df}[thm]{Definition}
\newtheorem{exa}[thm]{Example}

\newtheorem{rem}[thm]{Remark}

\title{Geometric duality, perfect graphs, and the Sierpi\'nski space}

\author[]{}

\author[P. Borodulin--Nadzieja]{Piotr Borodulin--Nadzieja}
\address[Piotr Borodulin--Nadzieja]{Mathematical Institute, University of Wroc\l aw, Pl. Grunwaldzki 2, 50-384 Wroc\l aw, Poland}
\email{pborod@math.uni.wroc.pl}

\author[B. Farkas]{Barnab\'as Farkas}
\address[Barnab\'as Farkas]{DMG/Algebra, TU Wien, Wiedner Hauptstrasse 8-10/104, 1040 Vienna, Austria}
\email{barnabasfarkas@gmail.com}

\author[A. Pelczar--Barwacz]{Anna Pelczar--Barwacz}
\address[Anna Pelczar-Barwacz]{Institute of Mathematics, Faculty of Mathematics and Computer Science, Jagiellonian University, \L ojasiewicza 6, 30-348 Krak\'ow, Poland}
\email{anna.pelczar@uj.edu.pl}

\thanks{B. Farkas was supported by the Austrian Science Fund (FWF) projects nos. I5918 and PAT5730424}


\begin{document}

\begin{abstract} 
In their classical paper \emph{On the stopping time Banach space}, Bang and Odell, among a plethora of results concerning the dyadic stopping time space and its dual, presented the first non-trivial example of the \emph{duality phenomenon} between
combinatorial Banach spaces. We give a full characterization of such pairs $(\mc{F}_0, \mc{F}_1)$ of families of finite sets: This duality holds iff there is a
perfect graph $G$ on $\NN$ such that $\mc{F}_0$ consists of all finite cliques of $G$ and $\mc{F}_1$ consists of all finite anti-cliques of $G$. As it turns out, Lov\'asz' famous perfect graph theorem is an immediate corollary of this result. Among the many  examples of such pairs of families, we investigate a particularly interesting
one, when $G$ is the Sierpi\'nski graph, and study general methods of embedding combinatorial and classical sequence spaces in the generated space, including the Schreier and $\ell_p$ spaces.
\end{abstract}

\maketitle

\section{Introduction}

Given a hereditary family $\mathcal{F}$ of finite subsets of $\mathbb{N}$ such that $\bigcup\mc{F}=\NN$, define the extended norm $\|\bullet\|_\mathcal{F}\colon \mathbb{R}^\mathbb{N} \to [0,\infty]$ and two associated spaces as follows:
\begin{align*} 
\| x \|_\mathcal{F}& = \sup\Big\{\sum_{i\in F} |x(i)|:F\in\mc{F}\Big\},\\
Y_\mathcal{F}& = \big\{x\in \mathbb{R}^\mathbb{N}: \| x \|_\mathcal{F} < \infty \big\}, \\
X_\mathcal{F}& = \big\{x\in \mathbb{R}^\mathbb{N}: \|P_{[n,\infty)}(x)\|_\mathcal{F} \to 0\big\} 
\end{align*}
where $P_A:\mbb{R}^\NN\to\mbb{R}^\NN$ is the usual projection along $A\subseteq\NN$.
Then $X_\mc{F}\subseteq Y_\mc{F}$ equipped with $\|\bullet\|_\mc{F}$ are Banach spaces, $X_\mathcal{F}$ is the completion of $c_{00}$, and the canonical algebraic basis $(e_n)$ of $c_{00}$ is a $1$-unconditional basis in $X_\mc{F}$. Spaces of the form $X_\mc{F}$ are called \emph{combinatorial (Banach) spaces}.\footnote{In some papers, because of the interplay between these spaces and analytic P-ideals, $X_\mc{F}$ is denoted by $\mrm{EXH}(\mc{F})$ and $Y_\mc{F}$ by $\mrm{FIN}(\mc{F})$. As here we do not deal with ideals, we will use the shorter notations.} For a survey on these spaces see Borodulin--Nadzieja, Farkas, Jachimek, and Pelczar--Barwacz \cite{ZOO}. 

\smallskip 
For example, let $[\NN]^{\leq n}=\{F\subseteq\NN:|F|\leq n\}$ and $[\NN]^{<\infty}=\{E\subseteq\NN:|E|<\infty\}$, then \[ X_{[\NN]^{\leq 1}}=c_0,\,Y_{[\NN]^{\leq 1}}=\ell_\infty,\,\text{ and }\,X_{[\NN]^{<\infty}}=Y_{[\NN]^{<\infty}}=\ell_1.\] Apart from $c_0$ and $\ell_1$, classical and well-studied examples of combinatorial spaces include the following: The \emph{Schreier space} $X_\mc{S}$ generated by the \emph{Schreier family} 
\[ \mc{S}=\big\{S\in [\NN]^{<\infty}:S=\0\,\text{ or }\,|S|\leq \min S\big\}\] 
and its higher order siblings (see later). Also, the \emph{Rosenthal space} $X_\mc{C}$ and the \emph{stopping time space} $X_\mc{A}$ where 
\begin{align*}
\mc{C} &=\big\{C\subseteq 2^{<\NN}:C\,\text{ is a finite chain}\big\},\\
\mc{A} &=\big\{A\subseteq 2^{<\NN}:A\,\text{ is a finite anti-chain}\big\}
\end{align*} 
and, of course, we define the associated spaces over the dyadic tree $2^{<\NN}$ instead of $\NN$. For a survey on these spaces see Bang and Odell \cite{BangOdell} where, among other results, it was proved that $X_\mc{C}$ is universal among separable Banach spaces with unconditional bases, that is, $X_\mc{C}$ contains (isomorphic) copies of all such spaces; and that $X_\mc{A}$ contains copies of all $\ell_p$ spaces ($1\leq p<\infty$).

\smallskip
The spaces of the form $Y_\mathcal{F}$ have been studied less intensively but we know that their structure is closely connected with $X_\mc{F}$: Considering $X_\mc{F}\subseteq X_\mc{F}^{**}$ and applying some results from Apatsidis \cite{Apatsidis}, it turns out that for every $y\in Y_\mc{F}$ the series $\sum_{n=1}^{\infty}y(n)e_n$ is weak$^*$ convergent in $X_\mc{F}^{**}$, moreover, the map 
\[ Y_\mc{F}\to X_\mc{F}^{**},\;y\mapsto w^*\!\!\lim_{n\to\infty}\sum_{k=1}^ny(k)e_k\] is an isometric isomorphism between $Y_\mc{F}$ and the weak$^*$ sequential closure of $X_\mc{F}$ in $X_\mc{F}^{**}$. As this isomorphism restricted to $X_\mc{F}$ is the identity, we may always consider 
\[ X_\mc{F}\subseteq Y_\mc{F}\subseteq X_\mc{F}^{**}.\] 
The two extremes, namely, equality on the left or right sides above can be characterized as follows (see Borodulin--Nadzieja and Farkas \cite{BNF20}): 
Given an $\mc{F}$ as above, the following are equivalent:
\begin{itemize}\setlength\itemsep{0.1cm}
\item[(a)] $X_\mc{F}=Y_\mc{F}$, 
\item[(b)] $X_\mc{F}$ and $Y_\mc{F}$ are isomorphic, 
\item[(c)] $X_\mc{F}$ does not contain copies of $c_0$;
\end{itemize}
and similarly, the following are also equivalent:
\begin{itemize}\setlength\itemsep{0.1cm}
\item[(a')] $Y_\mc{F}=X_\mathcal{F}^{**}$,
\item[(b')] $Y_\mc{F}$ and $X_\mathcal{F}^{**}$ are isomorphic, 
\item[(c')] $X_\mc{F}$ does not contain copies of $\ell_1$.
\end{itemize}

The work behind this paper was motivated by the observation that sometimes the dual space of $X_\mathcal{F}$, considering $X_\mc{F}^*\subseteq\mbb{R}^\NN$ along $(e^*_n)$, coincides with $Y_\mathcal{G}$ for some other family $\mathcal{G}$. This can be observed e.g. in the following cases (for $\mc{A}$ and $\mc{C}$ see \cite{BangOdell}):
\begin{multicols}{2}
\noindent\begin{align*}
	 X_{[\NN]^{\leq 1}}^* = c^*_0 &= \ell_1 = Y_{[\NN]^{<\infty}}\\
	 X_{[\NN]^{<\infty}}^* = \ell^*_1 &= \ell_\infty = Y_{[\NN]^{\leq 1}}
\end{align*}
\columnbreak
\begin{align*}
X_{\mc{C}}^* & =Y_{\mc{A}}\\
X_{\mc{A}}^* & =Y_{\mc{C}}
\end{align*}
\end{multicols}

We call this phenomenon \emph{geometric duality} (see later for the precise definitions) and we have found its rather surprising and purely combinatorial characterization (see Theorem \ref{conv}):

\begin{nnthm}\label{main}
The families $\mathcal{F}_0$ and $\mathcal{F}_1$ are geometrically dual if and only if there is a perfect graph $G$ on $\NN$ such that $\mathcal{F}_0=\{$finite cliques in $G\}$ and $\mathcal{F}_1=\{$finite anti-cliques in $G\}$.
\end{nnthm}

For example, if $G=(\NN,\0)$ is the graph on $\mathbb{N}$ without any edges, then the cliques are the singletons and every finite subset of $\mathbb{N}$ is an anti-clique. Also, if $G$ is the comparability graph of $2^{<\mathbb{N}}$, that is, $s$ is connected by an edge with $t$ iff $s$ extends $t$ or $t$ extends $s$, then the cliques coincide with the chains and the anti-cliques coincide with the anti-chains. 

At first sight, perfect graphs (see the definition below) seem to appear out of the blue in the context of geometric properties of Banach spaces. However, the connection between them and polytopes, crucial for the proof of Theorem \ref{conv}, was already present in the works of Chv\'atal \cite{Chvatal}, Fulkerson \cite{Fulkerson}, Lov\'asz \cite{Lovasz1,Lovasz2}, and others. 

\smallskip
Apart from the characterization theorem, the idea of generating Banach spaces from graphs led to many new interesting examples of combinatorial spaces.

\smallskip
The paper is organized as follows: 

\smallskip
In Section \ref{combspaces} we present a quick introduction to combinatorial spaces. 

\smallskip
In Section \ref{graph} we characterize families of the form $\{$finite (anti-)cliques of $G\}$ where $G$ is a graph on $\NN$ and provide a plethora of examples, including one (see Example \ref{perles}) witnessing that the Schur property does not imply the combinatorial Schur property (introduced in \cite{ZOO}). 

\smallskip
In Section \ref{geomdual} we prove the main theorem (Theorem \ref{conv}) connecting geometric duality and perfect
graphs, and including some further equivalent formulations of geometric duality. This result also yields a new proof of the perfect graph theorem: Proving Berge's conjecture, Lov\'asz (see \cite{Lovasz1}) showed that complements of perfect graphs are also perfect. We do not claim that we found a completely new proof because even Lov\'asz' original one relies on convex analysis of polytopes which is crucial in our proof as well. At the same time, the proof presented below is probably the first one of purely analytical provenience. We find it astonishing that this strictly combinatorial fact can be proven solely through Banach space theory. Another interesting consequence of the main theorem is the description of the extreme points of $B(Y_\mc{F})$ where $\mc{F}$ consists of the finite cliques of a perfect graph (see also Section
\ref{remext}).

\smallskip
In Section \ref{sier1} we study a, in a sense, canonical perfect graph, the \emph{Sierpi\'nski graph}: For a bijection $f:\mathbb{N} \to \mathbb{Q}$ let $n < m$ be connected by an edge if $f(n) < f(m)$, and let
$\mc{C}_f$ be the family of its finite cliques. Analogous objects, defined on uncountable subsets of $\mathbb{R}$, are intensively studied in Ramsey theory and in set theory of the reals. It turns out (see Proposition \ref{nonisom} and Theorem
\ref{permisom}) that, though graphs of this form are isomorphic only in trivial cases, the generated \emph{Sierpi\'nski spaces}, $X_{\mc{C}_f}$, are all isomorphic. Also, it follows that $X_{\mathcal{C}_f}^*$ and $Y_{\mathcal{C}_f}$ are also
isomorphic which, in the light of $X_{\mc{C}_f}\subseteq Y_{\mc{C}_f}\subseteq X_{\mc{C}_f}^{**}$, is somewhat awkward. These properties and the fact (see the later sections) that the Sierpi\'nski space contains copies of many classical sequence spaces make it a peculiar object with rich structure. At the same time, it does not contain copies of e.g. $c_0(\ell_1)$ (see
Theorem \ref{c_0-l_1}), in particular, it is not universal in the class of separable Banach spaces with unconditional bases.

\smallskip
In Section \ref{emulations} a general method of embedding combinatorial spaces isometrically in the Sierpi\'nski space is presented. In particular, we show that this method applies to all Schreier spaces of finite order; also, that if we just slightly modify the definition of the Schreier sequence, then all the resulting Schreier$^*$ spaces embed isometrically in the Sierpi\'nski space.

\smallskip
In Section \ref{sier2} we prove that the Sierpi\'nski space contains (isomorphic) copies of every $\ell_p$ space ($1\leq p<\infty$). In fact, we prove a more general result, namely, that all reflexive Musielak-Orlicz spaces embed in the Sierpi\'nski space. 


\smallskip
In Section \ref{remext} we discuss some partial results regarding the extreme points of $B(Y_\mc{F})$ where $\mc{F}$ consists of all finite cliques of a graph $G$ on $\NN$. Applying the main theorem from Section \ref{geomdual} (see above), if $G$ is perfect, then these extreme points are of very simple, canonical form (see Remark \ref{extYF}). In the case of non-perfect graphs, this question seems to boil down to the analysis of the extreme points in the finite-dimensional subspaces generated by the holes and anitholes of $G$.

\subsection*{Acknowledgment} We would like to thank Sebastian Jachimek for the fruitful discussions and for his permission to enclose some of his results in Section \ref{remext}, Bal\'azs Keszegh for his guidance in the world of finite (hyper)graphs, and Jordi Lopez-Abad for the stimulating discussions over the subject of this paper. 

\section{Combinatorial Banach spaces}\label{combspaces}

We begin with a quick overview of these spaces. In general, we define combinatorial spaces over arbitrary non-empty countable, perhaps finite, $\Omega$ as follows (see \cite{BNF20,ZOO}): Fix a family $\mc{F}\subseteq[\Omega]^{<\infty}$ covering $\Omega$ and let $\|x\|_\mc{F}=\sup\{\sum_{k\in F}|x(k)|:F\in\mc{F}\}$ for $x\in\mbb{R}^\Omega$ and $Y_\mc{F}=\{x\in\mbb{R}^\Omega:\|x\|_\mc{F}<\infty\}$; to define $X_\mc{F}$ for an infinite $\Omega$, fix an enumeration $\Omega=\{a_n:n\in\NN\}$ and set 
\[ X_\mc{F}=\big\{x\in\mbb{R}^\Omega:\|P_{\{a_n,a_{n+1},\dots\}}(x)\|_\mc{F}\xrightarrow{n\to\infty} 0\big\}.\]
Obviously, the enumeration itself does not play any role here because 
\[X_\mc{F}=\big\{x\in\mbb{R}^\Omega:\inf\big\{\|P_{\Omega\setminus E}(x)\|_\mc{F}:E\in [\Omega]^{<\infty}\big\}=0\big\}.\]
The properties mentioned in the introduction still hold: $X_\mc{F}\subseteq Y_\mc{F}$ equipped with $\|\bullet\|_\mc{F}$ are Banach spaces, $X_\mc{F}$ is the completion of $c_{00}(\Omega)$, and the canonical algebraic basis $(e_n)=(e_n)_{n\in\Omega}$ of $c_{00}(\Omega)$ is a normalized $1$-unconditional basis in $X_\mc{F}$. When it is necessary, we will write $(e_{\mc{F},n})$ to clarify the underlying space. 

Notice that (i) we can and will always assume that $\mc{F}$ is hereditary, that is, $E\subseteq F\in\mc{F}$ implies $E\in\mc{F}$, and (ii) working with an arbitrary $\mc{H}\subseteq \mc{P}(\Omega)$ covering $\Omega$ does not lead to more general spaces because $\mc{H}'=\bigcup\{[H]^{<\infty}:H\in\mc{H}\}\subseteq [\Omega]^{<\infty}$ and $\|x\|_\mc{H}=\|x\|_{\mc{H}'}$ for every $x\in\mbb{R}^\Omega$.

\smallskip
We will work, for example, over $\Omega=\{1,\dots,n\},\mbb{N},\mbb{N}\times\mbb{N}$, and the dyadic tree
\[2^{<\mbb{N}}=\big\{s:s=\0\,\text{ or }\,s\,\text{ is a }\,\{1,2,\dots,n\}\to\{0,1\}\,\text{function}\,(n\in\NN)\big\}\]
but for now, as long as the underlying set does not play any crucial role, let us continue over $\Omega=\NN$ and set
\[ \mrm{FHC}=\mrm{FHC}(\NN)=\big\{\mc{F}\subseteq[\NN]^{<\infty}:\mc{F}\,\text{ is hereditary and covers }\,\NN\big\}.\]
When referring to topological properties of and operations on subsets of $\mc{P}(\NN)$, e.g. on elements of $\mrm{FHC}$, we consider $\mc{P}(\NN)\simeq 2^\mbb{N}$ equipped with the usual product topology. For example, the closure of an $\mc{F}\in\mrm{FHC}$ with respect to this topology is 
\[ \overline{\mc{F}}=\mc{F}\cup\big\{A\in [\NN]^{\infty}:[A]^{<\infty}\subseteq\mc{F}\big\}\]
where, of course, $[\NN]^\infty=\{A\subseteq\NN:|A|=\infty\}$.

\smallskip
Notice that an $\mc{F}\in\mrm{FHC}$ is compact (that is, $\overline{\mc{F}}=\mc{F}$) iff $\overline{\mc{F}}\subseteq [\NN]^{<\infty}$ iff $\mc{F}$ does not contain infinite $\subseteq$-chains. For example, $[\NN]^{\leq 1}$ and $\mc{S}$ are compact whereas $[\mbb{N}]^{<\infty}$, $\mc{C}$, and $\mc{A}$ are not. Combining some classical results (see \cite{BNF20}), we obtain that $\mc{F}$ is compact iff $X_\mc{F}$ does not contain copies of $\ell_1$ (see above) iff $X_\mc{F}$ is $c_0$-saturated. Compact families and combinatorial spaces generated by them have been extensively studied, see e.g. \cite{Jordi-Stevo, ExtremeKevin2, Brech-isometries, Brech-BanachStone, BNF20}. 

\smallskip
We will consider $X^*_\mc{F}\subseteq\mbb{R}^ \NN$ along the bi-orthogonal functionals $(e^*_n)$, that is, via $X^*_\mc{F}\ni \al\longleftrightarrow (\al(e_n))\in\mbb{R}^ \NN$. Now, if $\al\in B(X^*_\mc{F})$ then $|\al(n)|\leq 1$ for every $n$; conversely, if $\supp(\al)\in\overline{\mc{F}}$ and $|\al(n)|\leq 1$ for every $n$, then $\al\in B(X^*_\mc{F})$. In particular, if $\sigma\in \{\pm 1,0\}^ \NN$ is such that $\supp(\sigma)\in\overline{\mc{F}}$, then $\sigma\in B(X^*_\mc{F})$ and, unless $\supp(\sigma)=\0$, $\|\sigma\|^*_\mc{F}=1$ (where $\|\bullet\|^*_\mc{F}$ stands for the usual norm on $X^*_\mc{F}$). For $\mc{H}\subseteq\mc{P}(\NN)$ define
\[ W(\mc{H})=\big\{\sigma\in\{\pm 1,0\}^\mbb{N}:\supp(\sigma)\in \mc{H}\big\}.\]
Clearly, if $\mc{H}\subseteq\overline{\mc{F}}$ then $W(\mc{H})\subseteq B(X^*_\mc{F})$, if additionally every $F\in\mc{F}$ can be covered by an element of $\mc{H}$, then (see \cite[Lemma 4]{Brech-isometries}) $W(\mc{H})$ is a norming set and hence $\overline{\mrm{conv}}^{w^*}\!(W(\mc{H}))=B(X^*_\mc{F})$. Also, it is easy to check that the weak$^*$ topology on $W(\overline{\mc{F}})$ coincides with the topology inherited from $\{\pm 1,0\}^\mbb{N}$, and it follows that $(W(\overline{\mc{F}}),w^*)$ is compact and $X_\mc{F}\to C(W(\overline{\mc{F}}),w^*)$, $x\mapsto(\sigma\mapsto \sigma(x))$ is an isometric embedding.

\smallskip
The point is that with the help of the operation $W(\bullet)$ we obtain a clear description of the extreme points of $B(X^*_\mc{F})$: Let $\max(\overline{\mc{F}})$ be the set of all maximal elements of $(\overline{\mc{F}},\subseteq)$, that is, 
\[ \max(\overline{\mc{F}})=\big\{A\in\overline{\mc{F}}:A\,\text{ has no proper extension in }\overline{\mc{F}}\big\}.\] 

\begin{prop}\label{extremedual} \emph{(see \cite[Proposition 2.4]{ZOO})} $\mrm{Ext}(B(X_\mc{F}^*))=W(\max(\overline{\mc{F}}))$. 
\end{prop}

Let us point put out that describing the extreme points of $B(X_\mc{F})$ is a significantly more difficult problem, even in the case of the very well-studied Schreier space, more or less the only thing we know, see \cite{countext}, is that $\mrm{Ext}(B(X_\mc{S}))$ is countable. For the description of $\mrm{Ext}(B(X_\mc{F}))$ in several classes of combinatorial spaces see \cite[Chapter 4]{SebTh} and Section \ref{remext} below.

\smallskip
In the next section, in a more specified context, we will present many interesting examples of families with the combinatorial spaces they generate.

\section{Graph-generated families}\label{graph}

In the introduction we mentioned two important examples, $\mc{C}=\{C\subseteq 2^{<\NN}:C$ is a finite chain$\}$ and $\mc{A}=\{A\subseteq 2^{<\NN}:A$ is a finite anti-chain$\}$ where $2^{<\NN}$ is considered with the usual partial order, $s\unlhd t$ iff $t$ extends $s$. Of course, we can work with an arbitrary countable poset $\PP=(P,\leq)$ and define $\mc{C}(\PP)=\{C\subseteq P:C$ is a finite chain in $\PP\}$ and $\mc{A}(\PP)=\{A\subseteq P:A$ is a finite anti-chain in $\PP\}$. Moreover, we can go one crucial step further: Given a non-empty countable $V$ and a (simple) graph $G=(V,E)$, that is, $E\subseteq [V]^2=\{\{a,b\}:a,b\in V,a\ne b\}$, define  
\begin{align*} 
\mc{C}(G)& =\big\{C\in [V]^{<\infty}:C\,\text{ is a \emph{clique} in $G$, i.e. }\,[C]^2\subseteq E\big\},\\
\mc{A}(G)& =\big\{A\in [V]^{<\infty}:A\,\text{ is an \emph{anti-clique} in $G$, i.e. }\,[A]^2\cap E=\0\big\}.
\end{align*}
Given a poset $\PP=(P,\leq)$, let $G_\PP$ be its \emph{comparability graph}, that is, $G_\PP$ lives on $P$ and $\{p,q\}$ is an edge iff $p<q$ or $q<p$. It follows that $\mc{C}(\PP)=\mc{C}(G_\PP)$ and $\mc{A}(\PP)=\mc{A}(G_\PP)$. 

\smallskip
Obviously, if $C\in\mc{C}(G)$ and $A\in\mc{A}(G)$ then $|C\cap A|\leq 1$, moreover, 
\begin{align*}
\mc{C}(G)&=\big\{C\in [V]^{<\infty}:|C\cap A|\leq 1\,\text{ for every }\,A\in\mc{A}(G)\big\}\\
\mc{A}(G)&=\big\{A\in [V]^{<\infty}:|A\cap C|\leq 1\,\text{ for every }\,C\in\mc{C}(G)\big\}.
\end{align*} 

This observation leads to the following construction: Given $\mc{F}\in\mrm{FHC}$ define 
\[ \mc{F}^\perp=\big\{E\in [\NN]^{<\infty}:|E\cap F|\leq 1\,\text{ for every }\,F\in\mc{F}\big\}.\]
Then $\mc{F}^\perp\in\mrm{FHC}$ as well. In general, $\mc{F}^{\perp\perp}\supseteq\mc{F}$ always holds but $\mc{F}^{\perp\perp}$ can be much larger than $\mc{F}$, e.g. with $\mc{S}=\{S\in [\NN]^{<\infty}:S=\0$ or $|S|\leq\min S\}$ we have
\[ \mc{S}^\perp=[\NN]^{\leq 1}\cup\big\{\{1,n\}:n\in\NN\big\}\,\text{ and }\,\mc{S}^{\perp\perp}=\big\{\{1\}\big\}\cup\big[\NN\setminus\{1\}\big]^{<\infty}.\] 
At the same time, $\mc{C}(G)^\perp=\mc{A}(G)$ and $\mc{A}(G)^\perp=\mc{C}(G)$, hence $\mc{C}(G)^{\perp\perp}=\mc{C}(G)$ and $\mc{A}(G)^{\perp\perp}=\mc{A}(G)$ for every graph $G$, moreover, the following holds:  

\begin{fact}\label{graphgen}
Given $\mc{F}\in\mrm{FHC}$, the following are equivalent:
\begin{itemize}\setlength\itemsep{0.1cm}
\item[(i)] $\mc{F}=\mc{E}^\perp$ for some $\mc{E}\in\mrm{FHC}$.
\item[(ii)] $\mc{F}=\mc{F}^{\perp\perp}$.
\item[(iii)] $\mc{F}=\mc{C}(G)$ for some graph $G$ on $\NN$.
\end{itemize}
\end{fact}
\begin{proof}
(i)$\to$(ii): We know that $\mc{E}^{\perp\perp\perp}\supseteq\mc{E}^\perp$, conversely, if $H\in\mc{E}^{\perp\perp\perp}$ then $|H\cap E|\leq 1$ for every $E\in\mc{E}^{\perp\perp}\supseteq\mc{E}$, hence $H\in\mc{E}^\perp$.

\smallskip
(ii)$\to$(iii): Let $G=(\NN,[\NN]^2\cap\mc{F})$. Obviously, $\mc{F}\subseteq\mc{C}(G)$. Conversely, let $C\in\mc{C}(G)$ and assume on the contrary that $C\notin\mc{F}=\mc{F}^{\perp\perp}$. Then there is an $A\in\mc{F}^\perp$ such that $|C\cap A|>1$. In other words, $C$ contains a pair $\{a,b\}\in\mc{F}^\perp$ but this is impossible because $\{a,b\}$ is an edge and hence $\{a,b\}\in\mc{F}$. 

\smallskip
(iii)$\to$(i): We know that $\mc{C}(G)=\mc{A}(G)^\perp$. 
\end{proof}

\begin{df}
An $\mc{F}\in\mrm{FHC}$ is \emph{graph-generated} if $\mc{F}=\mc{C}(G)$ for some (unique) graph $G$ on $\NN$. 
\end{df}

Clearly, if $\mc{F}=\mc{C}(G)=\mc{A}(G^c)$ is graph-generated, then so is $\mc{F}^\perp=\mc{A}(G)=\mc{C}(G^c)$ where $G^c=(V,[V]^2\setminus E)$ is the \emph{complement} of $G$.

\begin{exa}[$\bigcup\mc{Q}$ and $\mrm{Fh}(\mc{Q})$]\label{farah} 
Obviously, $[\NN]^{\leq 1}=\mc{C}(\NN,\0)$ and $[\NN]^{<\infty}=\mc{C}(\NN,[\NN]^2)$ are graph-generated. The simplest generalizations of these examples are the following constructions: Let $(V_n)$ be a partition of $\NN$ into non-empty finite sets
and $\mc{Q}=(\mc{Q}_n)$ a sequence of families, $\mc{Q}_n\in\mrm{FHC}(V_n)$. Then we can define $\bigcup\mc{Q}:=\bigcup_n\mc{Q}_n\in\mrm{FHC}$ and  the \emph{Farah family generated by} $\mc{Q}$ (see \cite{ZOO}) 
\[ \mathrm{Fh}(\mc{Q}) = \big\{F \in [\NN]^{<\infty}\colon F\cap V_n \in \mc{Q}_n \mbox{ for every }\, n\big\}. \]
Notice that the Farah family generated by $\mc{Q}^\perp = (\mc{Q}^\perp_n)$ coincides with $(\bigcup \mc{Q})^\perp$ 
and vice versa, $\mrm{Fh}(\mc{Q})^\perp=\bigcup\mc{Q}^\perp$. Clearly, $\bigcup\mc{Q}$ is graph-generated iff $\mrm{Fh}(\mc{Q})$ is graph-generated iff every $\mc{Q}_n$ is graph-generated. 

Regarding the generated combinatorial spaces, for arbitrary $\mc{Q}$, $X_{\bigcup\mc{Q}}$ is the $c_0$-product of the sequence $(X_{\mc{Q}_n})$, $Y_{\bigcup\mc{Q}}$ is the $\ell_\infty$-product of this sequence, and $X_{\mrm{Fh}(\mc{Q})}=Y_{\mrm{Fh}(\mc{Q})}$ is the $\ell_1$-product of this sequence.
\end{exa}

\begin{exa}[comparability and co-comparability graphs]\label{comcocom}  Many important examples of graph-generated families are generated by comparability graphs. For example, such graphs are $(\NN,\0)=G_{(\NN,\Delta)}$ where $\Delta=\{(n,n):n\in\NN\}$ is the trivial partial order and its complement $(\NN,[\NN]^2)=G_{(\NN,\leq)}$ where $\leq$ stands for the usual ordering on $\NN$, the associated families are $[\NN]^{\leq 1}=\mc{C}(\NN,\Delta)$ and $[\NN]^{<\infty}=\mc{C}(\NN,\leq)$. It turns out that the complement of $G_{(2^{<\NN},\unlhd)}$ is also a comparability graph, and hence 
\[ \mc{A}(2^{<\NN},\unlhd)=\mc{A}(G_{(2^{<\NN},\unlhd)})=\mc{C}(G^c_{(2^{<\NN},\unlhd)})=\mc{C}(G_{(2^{<\NN},\preceq)})=\mc{C}(2^{<\NN},\preceq)\] for some partial order $\preceq$ on $2^{<\NN}$: For $s,t\in 2^{<\NN}$ define $s\preceq t$ iff $s=t$ or there is a $k\in \dom(s)\cap\dom(t)$ where $s(k)\ne t(k)$ and $s(k)<t(k)$ at the first such $k$ (notice that $\0$ is an `isolated point' with respect to $\preceq$). Then $\preceq$ is indeed a partial order on $2^{<\NN}$ and $G_{(2^{<\NN},\preceq)}=G^c_{(2^{<\NN},\unlhd)}$.

In other words, $G_{(2^{<\NN},\unlhd)}$ is both comparability and co-comparability.\footnote{Probably the simplest example of a graph which is neither comparability nor co-comparability is the $5$-long cycle $C_5$, and the simplest comparability graph which is not co-comparability is $C_6$.} Such graphs have an interesting characterization (see \cite{2dimposet}): A graph $G$ is both comparability and co-comparability iff there is a partial order $(P,\leq)$ such that $G=G_{(P,\leq)}$ and $\leq$ is the intersection of two linear orders on $P$. For example, $\unlhd=(\unlhd\cup\preceq)\cap(\unlhd\cup\succeq)$. 
\end{exa}

\begin{exa}[Perles graph]\label{perles} 
Consider the product of two copies of $(\NN,\leq)$ as a product poset, that is, define $\leq$ on $\NN\times\NN$ as $(n_0,m_0)\leq (n_1,m_1)$ iff $n_0\leq n_1$ and $m_0\leq m_1$. First of all, notice that $\leq$ is also the intersection of two linear orders: Let $(n_0,m_0)\leq_x(n_1,m_1)$ iff $n_0<n_1$ \underline{or} $n_0=n_1$ and $m_0\leq m_1$, and similarly, let $(n_0,m_0)\leq_y(n_1,m_1)$ if $m_0<m_1$ \underline{or} $m_0=m_1$ and $n_0\leq n_1$. Clearly, $\leq_x$ and $\leq_y$ are linear orders and $\leq=\leq_x\cap\leq_y$. 

We know (see \cite{Perles} for more general results) that $(\NN^2,\leq)$ contains arbitrary large finite anti-chains, e.g. $\{(1,n),(2,n-1),\dots,(n,1)\}\in\mc{A}_2=\mc{A}(\NN^2,\leq)$ but it does not contain infinite anti-chains: Indeed, assume that $A=\{(n_k,m_k):k\in\NN\}$ is an infinite anti-chain. Then $n_k\ne n_{k'}$ for $k\ne k'$, hence we can assume that $n_1<n_2<\dots$ but then $m_1>m_2>\dots$, a contradiction. It follows that $\mc{A}_2$ is compact i.e. $X_{\mc{A}_2}$ is $c_0$-saturated
\end{exa}

Regarding the last example, we are more interested in the space $X_{\mc{C}_2}$ where $\mc{C}_2=\mc{C}(\mbb{N}^2,\leq)$ because, as we will show below, it is an example of a combinatorial space distinguishing two properties which are rather close to each other. Recall that a Banach space has a Schur property if each weakly converging sequence converges in norm. For a combinatorial space $X_\mc{F}$ this holds  (see \cite[Theorem 3.3]{ZOO}) iff
\begin{itemize}\setlength\itemsep{0.1cm}
	\item[$(S)$] for every $\mc{F}$-supported normalized block basic sequence $(x_n)$ in $X_{\mc{F}}$ there is a $D\in\overline{\mc{F}}$ such that $\limsup_{n\to\infty}\|P_D(x_n)\|_\mc{F}>0$.
\end{itemize}
We say that $X_\mc{F}$ has combinatorial Schur property (defined only for combinatorial spaces, see \cite[page 109]{ZOO}) if
\begin{itemize}\setlength\itemsep{0.1cm}
\item[$(S^*)$] for every pairwise disjoint $(F_n)$ in $\mc{F}$ there is a finite $\mc{D}\subseteq\overline{\mc{F}}$ such that $\bigcup\mc{D}$ covers infinitely many $F_n$.
\end{itemize}

This is interesting because $(S^*)$ clearly implies $(S)$ and so far no countexample for the reverse implication was known.

\begin{prop} The family $\mc{C}_2$ satisfies $(S)$ but not $(S^*)$.
\end{prop}
\begin{proof}
$(S)$ holds: Fix a $\mc{C}_2$-supported normalized block basic sequnce $x_n$, that is, $\supp(x_n)$ are pairwise disjoint finite chains in $(\NN^2,\leq)$ and $\|x_n\|=\|x_n\|_{\mc{C}_2}=1$ for every $n$. 
We can assume that neither $D_m=\{m\}\times\NN$ nor $D^m=\NN\times\{m\}$ work as $D$ in the definition of $(S)$ for any $m$, that is, 
\[\tag{$\sharp$} \|P_{D_m}(x_n)\|\xrightarrow{n\to\infty}0\,\text{ and }\,\|P_{D^m}(x_n)\|\xrightarrow{n\to\infty}0\,\text{ for every $m$.}\] 
By recursion we will construct $n_1<n_2<\dots$ and $C_k\subseteq \supp(x_{n_k})$ such that $D=\bigcup_kC_k$ is an infinite chain in $(\NN^2,\leq)$ and $\|P_D(x_{n_k})\|=\|P_{C_k}(x_{n_k})\|>1-2^{-k}$.

Let $n_1=1$, $C_1=\supp(x_1)$, and assume the we have already constructed $n_k$ and $C_k$ such that $\|P_{C_m}(x_{n_m})\|>1-2^{-m}$ for every $m=1,\dots,k$ and $C'_k=C_1\cup C_2\cup\dots\cup C_k\in\mc{C}_2$. Fix an $M$ such that $C'_k\subseteq \{1,\dots,M\}^2$ and let $H=D_1\cup \dots\cup D_M\cup D^1\cup\dots\cup D^M$. Applying $(\sharp)$, if $n$ is large enough, then $\|P_H(x_n)\|<2^{-k-1}$, pick such an $n_{k+1}>n_k$ and let $C_{k+1}=\supp(x_{n_{k+1}})\setminus H$. 

\smallskip
$(S^*)$ fails: Consider $C_n=\{n\}\times \{1,2,\dots,n\}\in\mc{C}_2$, $n\in\NN$. Given an infinite $\{n_1<n_2<\dots\}\subseteq\NN$, as $\bigcup_kC_{n_k}$ contains arbitrary long anti-chains, e.g. 
\[\big\{(n_m,m),(n_{m+1},m-1),(n_{m+2},m-2),\dots,n_{2m-1},1)\big\},\] 
$\bigcup_kC_{n_k}$ cannot be cover by finitely many elements of $\overline{\mc{C}}_2$.
\end{proof}

\begin{exa}[Sierpi\'nski graph]\label{sierpinski} 
Consider the intersection of $\NN$'s order and (a copy of) $\mbb{Q}$'s order: Fix a bijection $f\colon \mathbb{N} \to \mathbb{Q}$ and define $n\leq_f k$ iff $n\leq k$ and $f(n)\leq f(k)$. Then $\leq_f$ is a partial order on $\NN$ and its comparability graph $G_f=G_{(\NN,\leq_f)}$ is also a co-comparability graph (see Example \ref{comcocom}).  

It follows that $\mc{C}(G_f)$ consists of those sets $\{n_1 < \dots < n_k\}$ for which $f(n_1) < \dots < f(n_k)$ and $\mc{A}(G_f)$ of those sets whose `image' is decreasing. The idea of this example is based on Sierpi\'nski's coloring from Ramsey theory (see \cite{sier33}). We will investigate the spaces of the form $X_{\mc{C}(G_f)}$ in Section \ref{sier1}. 
\end{exa}

\section{Geometric duality}\label{geomdual}

As stated in \cite{BangOdell}, there is a duality between the combinatorial spaces induced by $\mc{A}=\{$finite anti-chains in $2^{<\NN}\}$ and $\mc{C}=\{$finite chains in $2^{<\NN}\}$, namely, 
\[ X_\mathcal{A}^*=Y_\mathcal{C}\,\text{ and }\,X_\mathcal{C}^*=Y_\mathcal{A}\]
where, as mentioned above, we consider $X_\mc{F}^*\subseteq\mbb{R}^\NN$ along the coordinate-functionals $(e^*_{\mc{F},n})$, and  equality between these spaces means that the norms coincide. Equivalently (see below), the basic sequences $(e^*_{\mc{A},n})$ and $(e_{\mc{C},n})$ are isometrically equivalent, just like $(e^*_{\mc{C},n})$ and $(e_{\mc{A},n})$. In this section we study this phenomenon and characterize all such pairs of families. Isometrical equivalence of basic sequences will sometimes be denoted by $(a_n)\stackrel{i}{\sim}(b_n)$. 

\begin{df} We say that $\mc{F}\in\mrm{FHC}$ is the \emph{geometric dual} of $\mc{H}\in\mrm{FHC}$ if $Y_\mc{F}=X^*_\mc{H}$.
\end{df}

First of all, we reformulate the definition of duality in terms of the canonical basic sequences, show that orthogonality of families is a necessary condition of geometric duality, and that this notion is symmetric:

\begin{prop}\label{geomsym}
Given $\mc{F},\mc{H}\in\mrm{FHC}$, $\mc{F}$ is the geometric dual of $\mc{H}$ iff $(e_{\mc{F},n})$ and $(e^*_{\mc{H},n})$ are isometrically equivalent. Furthermore, in this case, 
$\mc{H}=\mc{F}^\perp$ and $\mc{H}$ is the geometric dual of $\mc{F}$ as well, hence $\mc{F}=\mc{H}^\perp$ (and so $\mc{F}$ and $\mc{H}$ are graph-generated).
\end{prop}
\begin{proof}
To show the non-trivial implication in the first statement, assume that $(e_{\mc{F},n})\stackrel{i}{\sim}(e^*_{\mc{H},n})$. As \[ Y_\mc{F}=\Big\{y\in\mbb{R}^\NN:\sup_n\|P_{\{1,\dots,n\}}(y)\|_\mc{F}<\infty\Big\} \] and $\|y\|_\mc{F}=\sup_n\|P_{\{1,\dots,n\}}(y)\|_\mc{F}$, we need to show that 
\[ X^*_\mc{H}=\Big\{\al\in\mbb{R}^\NN:\sup_n\big\|P_{\{1,\dots,n\}}(\al)\big\|^*_\mc{H}<\infty\Big\}\,\text{ and }\,\|\al\|^*_\mc{H}=\sup_n\big\|P_{\{1,\dots,n\}}(\al)\big\|^*_\mc{H}.\]
Now, $X^*_\mc{H}$ is certainly contained in the set on the right side. Conversely, let $\al\in\mbb{R}^\NN$ be such that $\sup_n\|P_{\{1,\dots,n\}}(\al)\|^*_\mc{H}<\infty$, and let $\al_n=P_{\{1,\dots,n\}}(\al)$. Given an $x\in X_\mc{H}$, as
$(\al_n(x))$ is bounded, it has a convergent subsequence $\al_{n_k}(x)\xrightarrow{k\to\infty}a$, and it is straightforward to show that $\al_n(x)\xrightarrow{n\to\infty}a$ follows, that is, $\al(x)=a\in\mbb{R}$ is well-defined; and this implies that $\|\al\|^*_\mc{H}=\sup_n\|\al_n\|^*_\mc{H}$ also holds. 

\smallskip
$\mc{H}\subseteq\mc{F}^\perp$ because otherwise there is a pair $\{n,k\}\in\mc{H}\cap\mc{F}$ and hence $\|e_{\mc{F},n}+e_{\mc{F},k}\|_\mc{F}=2$ but $\|e^*_{\mc{H},n}+e^*_{\mc{H},k}\|^*_\mc{H}=1$. $\mc{F}^\perp\subseteq\mc{H}$ because if there is an $S\in\mc{F}^\perp\setminus\mc{H}$, then $\|\sum_{n\in S}e_{\mc{F},n}\|_\mc{F}=1$ but with $x=\chi_S/(|S|-1)\in B(X_\mc{H})$ (where $\chi_S$ stands for the characteristic function of $S$) we have $(\sum_{n\in S}e^*_{\mc{H},n})(x)=|S|/(|S|-1)$ and hence $\|\sum_{n\in S}e^*_{\mc{H},n}\|^*_\mc{H}>1$.

\smallskip
Before proving that $\mc{H}$ is the geometric dual of $\mc{F}$, let us show the following:
\begin{nnclaim}
Let $(b_n)$ be a basis in a Banach space $X$ with bi-orthogonal functionals $(b^*_n)$, and let $b^{**}_n\in [b^*_n]^*$ be the bi-orthogonal functional to $b^*_n$. Then $(b_n)$ and $(b_n)^{**}$ are isometrically equivalent.
\end{nnclaim}
\begin{proof}[Proof of the claim]
Fix a finitely supported $x=\sum_{n=1}^mc(n)b_n\in X$ and let $x^{**}=\sum_{n=1}^mc(n)b^{**}_n$. 

\smallskip
$\|x^{**}\|\leq \|x\|$: If $f\in B([b^*_n])$, then $|x^{**}(f)|=|f(x)|\leq \|f\| \|x\|\leq \|x\|$. 

\smallskip
$\|x\|\leq \|x^{**}\|$: Consider $x$ as a vector in a finite dimensional space $[b_1,\dots,b_m]$ and pick a functional $f\in B([b_1,\dots,b_m]^*)=B([b_1^*,\dots,b_m^*])$ with $\|x\|=|f(x)|$. By the Hahn-Banach theorem, extend $f$ to a functional $f'\in B(X^*)$ and note that $\|x^{**}\|\geq|x^{**}(f')|=|f'(x)|=|f(x)|=\|x\|$.
\end{proof}

Applying the claim, $(e_{\mc{H},n})\stackrel{i}{\sim}(e^{**}_{\mc{H},n})$; and combining this with $(e^*_{\mc{H},n})\stackrel{i}{\sim}(e_{\mc{F},n})$, we obtain that $(e_{\mc{H},n})\stackrel{i}{\sim}(e^{**}_{\mc{H},n})\stackrel{i}{\sim}(e^*_{\mc{F},n})$.
\end{proof}

At this point we have to recall some notions and notations from graph theory. Given an arbitrary graph $G=(V,E)$ and a non-empty $H\subseteq V$, let us denote $G\clrest H:=(H,E\cap [H]^2)$ the \emph{subgraph of $G$ induced by $H$}. The \emph{chromatic number} $\chi(G)$ of $G$ is the minimal number of anti-cliques required to cover $V$, and the \emph{clique number} $\om(G)$ of $G$ is the supremum of the sizes of its cliques. Obviously, $\chi(G)\geq \om(G)$. We say that $G$ is \emph{perfect} if $\chi(G\clrest H)=\om(G\clrest H)$ for every finite $H\subseteq G$. There are some other reasonable ways to define infinite perfect graphs but this one is exactly what we need. For a survey on many known classes of (finite) perfect graphs see \cite{pftclasses}, for now, mainly because of our examples in the previous section, let us only mention the following well-known result:

\begin{prop}\label{comppft} \emph{(Mirsky's / dual Dilworth theorem)} Comparability graphs are perfect.
\end{prop}

The study of perfect graphs has a long and incredibly rich history reaching its apex with the renowned strong perfect graph theorem:

\begin{thm}\label{strong-perfect} \emph{(Chudnovsky, Robertson, Seymour, and Thomas \cite{Chudnovsky})} A graph $G$ is perfect if, and only if, it contains neither odd holes, that is, induced subgraphs isomorphic to an odd cycle of length $\geq 5$, nor odd antiholes, that is, induced subgraphs isomorphic to the
complement of an odd cycle of length $\geq 5$. 
\end{thm}

At the same time, we will not need this canon below, for us the crucial result is Chv\'atal's following characterization:

\begin{thm}\label{stab} \emph{(see \cite{Chvatal})} A finite graph $G=(V,E)$ is perfect iff 
\[ \mrm{conv}\big\{\chi_A\in\mbb{R}^V:A\in\mc{A}(G)\big\}=\Big\{x\in[0,1]^V:\sum_{v\in C}x(v)\leq 1\,\text{ for every }\,C\in\mc{C}(G)\Big\}.\]
\end{thm}

It turned out that this result was an important ingredient in the proof of Lov\'asz' perfect graph theorem (see \cite{Lovasz1}), see also our remarks and Corollary \ref{perfect-graph-theorem} below.  

\smallskip
As stated in the introduction, the main result of this section is the equivalence between (0) and (5) in the next theorem.  
Given a family $\mc{F}\in\mrm{FHC}$ and a non-empty $Y\subseteq\NN$, let $\mc{F}\clrest Y=\mc{F}\cap\mc{P}(Y)=\{F\in\mc{F}:F\subseteq Y\}$.

\begin{thm}\label{conv} 
Given $\mc{F}\in\mrm{FHC}$, considering $X_{\mc{F}}\subseteq X^{**}_{\mc{F}}$ and the weak$^*$ topology on $X^{**}_\mc{F}$, the following are equivalent:
\begin{itemize}\setlength\itemsep{0.1cm}
\item[(0)] $\mc{F}$ and $\mc{F}^\perp$ are geometric duals.
\item[(1)] $\mrm{Ext}(B(X_{\mc{F}\upharpoonright V}))=W(\max(\mc{F}^\perp\clrest V))$ for every non-empty finite $V\subseteq\NN$.
\item[(2)] $B(X_{\mc{F}})\cap c_{00}=\mrm{conv}(W(\mc{F}^\perp))$.
\item[(3)] $B(X_{\mc{F}})=\overline{\mrm{conv}}^{\|\bullet\|_{\mc{F}}}(W(\mc{F}^\perp))$.
\item[(4)] $B(X^{**}_{\mc{F}})=\overline{\mrm{conv}}^{w^*}(W(\mc{F}^\perp))$. 

\item[(5)] $\mc{F}=\mc{C}(G)$ is generated by a perfect $G$. 
\end{itemize}
\end{thm}

Before proving the theorem, let us make some observations. As (0) is `symmetric' between $\mc{F}$ and $\mc{F}^\perp$ (see Proposition \ref{geomsym}), so are the rest of the statements, that is, switching the roles of $\mc{F}$ and $\mc{F}^\perp$ e.g. in (3) also leads to an equivalent statement, 
\begin{itemize}\setlength\itemsep{0.1cm}
\item[(3)$^\perp$] $B(X_{\mc{F}^\perp})=\overline{\mrm{conv}}^{\|\bullet\|_{\mc{F}^\perp}}(W(\mc{F}))$.
\end{itemize}

More importantly, as a by-product of this symmetry, we obtain a new proof of the perfect graph theorem: 

\begin{cor}\label{perfect-graph-theorem}\emph{(Perfect graph theorem)} If $G$ is perfect, then so is $G^c$.
\end{cor}
\begin{proof}[Proof of Corollary \ref{perfect-graph-theorem}]
Given a perfect graph $G$, (5)$^\perp$ says that $\mc{F}^\perp=\mc{C}(G)^\perp=\mc{C}(G^c)$ is generated by a perfect graph, in other words, that $G^c$ is perfect.
\end{proof}

\begin{proof}[Proof of Theorem \ref{conv}] 
We will prove (0)$\to$(1)$\to$(2)$\to$(0), (2)$\to$(3)$\to$(4)$\to$(2), and (2)$\leftrightarrow$(5). Regarding (1)$-$(4), clearly, $W(\mc{F}^\perp)\subseteq B(X_\mc{F})$ holds for every family $\mc{F}$.

\smallskip
(0)$\to$(1): Notice that (0) holds iff 
\[\tag{$0_V$} (e_{\mc{F},n})_{n\in V}\stackrel{i}{\sim} (e^*_{\mc{F}^\perp,n})_{n\in V}\;\text{ i.e. }\;B(X_{\mc{F}\upharpoonright V})=B(X^*_{\mc{F}^\perp\upharpoonright V})\] for every non-empty finite $V\subseteq \NN$, hence, in this case, applying Proposition \ref{extremedual}, $\mrm{Ext}(B(X_{\mc{F}\upharpoonright V}))=\mrm{Ext}(B(X^*_{\mc{F}^\perp\upharpoonright V}))=W(\max(\mc{F}^\perp\clrest V))$. 

\smallskip
(1)$\to$(2): As elements of $B(X_\mc{F})\cap c_{00}$ and $W(\mc{F}^\perp)$ are finitely supported and $\mc{F}^\perp$ is hereditary, (2) is equivalent to
\[\tag{$2_V$} B(X_{\mc{F}\upharpoonright V})=\mrm{conv}(W(\mc{F}^\perp\clrest V))\]
for every non-empty finite $V\subseteq \NN$. Fix such a $V$ and let $\mc{H}=\mc{F}\clrest V$, clearly, $\mc{H}^\perp=(\mc{F}\clrest V)^\perp=\mc{F}^\perp\clrest V$. Now, 
\begin{itemize}\setlength\itemsep{0.1cm}
\item[(i)] $B(X_\mc{H})=\mrm{conv}(\mrm{Ext}(B(X_\mc{H})))$ because $X_\mc{H}$ is finite dimensional,
\item[(ii)] $\mrm{conv}(\mrm{Ext}(B(X_\mc{H})))=\mrm{conv}(W(\max(\mc{H}^\perp)))$ because (1), and 
\item[(iii)] $\mrm{conv}(W(\max(\mc{H}^\perp)))=\mrm{conv}(W(\mc{H}^\perp))$ via trivial manipulations of signs.
\end{itemize}

\smallskip
(2)$\to$(0): Fix a non-empty finite $V\subseteq\NN$ and let $\mc{H}=\mc{F}\clrest V$, we show that $(0_V)$ holds, i.e. $B(X_\mc{H})=B(X^*_{\mc{H}^\perp})$.  Applying $(2_V)$ and (iii), $B(X_\mc{H})=\mrm{conv}(W(\mc{H}^\perp))=\mrm{conv}(W(\max(\mc{H}^\perp)))$ which is further equal to $\mrm{conv}(\mrm{Ext}(B(X^*_{\mc{H}^\perp})))$ (see Proposition \ref{extremedual}) and hence to $B(X^*_{\mc{H}^\perp})$ as in (i).

\smallskip 
(2)$\to$(3): $B(X_\mc{F})\cap c_{00}$ is norm-dense in $B(X_\mc{F})$.

\smallskip
(3)$\to$(4): $B(X_\mc{F})=\overline{\mrm{conv}}^{\|\bullet\|_{\mc{F}}}(W(\mc{F}^\perp))=\overline{\mrm{conv}}^{\|\bullet\|^{**}_\mc{F}}(W(\mc{F}^\perp))\subseteq \overline{\mrm{conv}}^{w^*}(W(\mc{F}^\perp))$ and we know (Goldstine's theorem) that $\overline{B(X_\mc{F})}^{w^*}=B(X^{**}_\mc{F})$.

\smallskip
(4)$\to$(2): Fix an $x\in B(X_\mc{F})\cap c_{00}$ and a sequence $x_n\in\mrm{conv}(W(\mc{F}^\perp))$ which is $w^*$-converging to $x$ in $X^{**}_\mc{F}$. If $\supp(x)\subseteq I:=\{1,\dots,m\}$, then, of course,
\[ P_{I}(x_n)\in P_I\big[\mrm{conv}(W(\mc{F}^\perp))\big]= \mrm{conv}\big(P_I[W(\mc{F}^\perp)]\big)\] also $w^*$-converges to $x=P_I(x)$, therefore, $x\in\overline{\mrm{conv}}^{w^*}(P_I[W(\mc{F}^\perp)])$. As $P_I[W(\mc{F}^\perp)]$ is finite, its convex hull is closed in any vector topology, hence $x\in\mrm{conv}(P_I[W(\mc{F}^\perp)])\subseteq\mrm{conv}(W(\mc{F}^\perp))$.

\smallskip
(2)$\leftrightarrow$(5): Theorem \ref{stab} says that (5) holds iff $\mc{F}$ is graph-generated and  
\[\tag{$5_V$}\mrm{conv}\big(\big\{x\in W(\mc{F}^\perp\clrest V):x\geq 0\big\}\big)=\big\{x\in B(X_{\mc{F}\upharpoonright V}):x\geq 0\big\}\]
for every non-empty finite $V\subseteq\NN$. As both $W(\mc{F}^\perp\clrest V)$ and $B(X_{\mc{F}\upharpoonright V})$ are closed for switching signs at coordinates, $(5_V)$ implies $(2_V)$. Conversely, we already know that (2) implies (0), hence $\mc{F}$ is graph-generated, we show that $(5_V)$ holds for every $V$. Let $\mc{H}=\mc{F}\clrest V$ and $W(\mc{H}^\perp)_+=\{x\in W(\mc{H}^\perp):x\geq 0\}$. For $E\subseteq V$ define $S_E:\mbb{R}^V\to\mbb{R}^V$, $S_E(x)(k)=-x(k)$ if $k\in E$ and $=x(k)$ otherwise. If we can show that
\[\tag{$\star$} \mrm{conv}\big(W(\mc{H}^\perp)\big)=\bigcup\big\{S_E\big[\mrm{conv}\big(W(\mc{H}^\perp)_+\big)\big]:E\subseteq V\big\}\]
holds (for every $\mc{H}$), then, applying (2), we are done. To show ($\star$), first of all, notice that $S_E[\mrm{conv}(W(\mc{H}^\perp)_+)]=\mrm{conv}(S_E[W(\mc{H}^\perp)_+])\subseteq \mrm{conv}(W(\mc{H}^\perp))$. To show the reverse inclusion, it is enough to check that the right side is convex. This follows from applying the following easy fact $|V|$ many times: Fix $v\in V$ and let $C\subseteq\{x\in\mbb{R}^V:x(v)\geq 0\}$ be a convex set such that $P_{V\setminus\{v\}}[C]\subseteq C$. Then $C\cup S_{\{v\}}[C]$ is convex.   
\end{proof}

\begin{df}
A family $\mc{F}\in\mrm{FHC}$ is called \emph{perfect}\footnote{One may find this notion confusing because of the topological property of the same name (see e.g. \cite{ZOO}), and we shall distinguish them e.g. by saying combinatorially / topologically perfect. At the same time, as the topological property will not come up in this paper at all, for now, we find it safe to use.}  if $\mc{F}=\mc{C}(G)$ is generated by a perfect graph $G$.   
\end{df}

\begin{rem}\label{extYF}
Applying Theorem \ref{conv} and Proposition \ref{extremedual}, given a perfect $\mc{F}$, we have a full description of the extreme points of the unite ball in $Y_\mc{F}$: 
\[ \mrm{Ext}\big(B(Y_\mc{F})\big)=\mrm{Ext}\big(B(X_{\mc{F}^\perp}^*)\big)=W\big(\max(\overline{\mc{F}^\perp})\big).\]
Of course, the smaller space $X_\mathcal{F}$ may have much fewer extreme points, consider e.g. $c_0\subseteq\ell_\infty$. 
\end{rem}

Let us take a look at our examples from Section \ref{graph} in the light of Theorem \ref{conv}:

\begin{exa}
It is trivial to check that disjoint unions of perfect graphs are perfect. Therefore, the following are equivalent: (i) $\bigcup\mc{Q}$ is perfect; (ii) $(\bigcup\mc{Q})^\perp$ is perfect; (iii) $\mc{Q}_n$ is perfect for every $n$; (iv) $\mc{Q}^\perp_n$ is perfect for every $n$; (v) $\mrm{Fh}(\mc{Q})$ is perfect; (vi) $\mrm{Fh}(\mc{Q})^\perp$ is perfect. If these families are perfect, then, applying Theorem \ref{conv}, $X_{\bigcup\mc{Q}}^*=Y_{\mrm{Fh}(\mc{Q}^\perp)}=X_{\mrm{Fh}(\mc{Q}^\perp)}$ and $X_{\mrm{Fh}(\mc{Q})}^*=Y_{\mrm{Fh}(\mc{Q})}^*=Y_{\bigcup\mc{Q}^\perp}$.  
\end{exa}

\begin{exa} Applying Proposition \ref{comppft} and the symmetry mentioned above, the families $\mc{C}(\PP)$ and $\mc{A}(\PP)$ are perfect for every poset $\PP$. This covers the following examples:
\begin{itemize}\setlength\itemsep{0.1cm}
\item $[\NN]^{\leq 1}$ and $[\NN]^{<\infty}$, that is, $c^*_0=\ell_1$ and $\ell^*_1=\ell_\infty$;
\item the original $\mc{C}=\mc{C}(2^{<\NN},\unlhd)$ and $\mc{A}=\mc{A}(2^{<\NN},\unlhd)$, hence $X_\mc{C}^*=Y_\mc{A}$ and $X_\mc{A}^*=Y_\mc{C}$;
\item the families $\mc{C}_2=\mc{C}(\NN^2,\leq)$ and $\mc{A}_2=\mc{A}(\NN^2,\leq)$ from Example \ref{perles}, in other words, $X_{\mc{C}_2}^*=Y_{\mc{A}_2}$ and $X_{\mc{A}_2}^*=Y_{\mc{C}_2}$;
\item the families $\mc{C}(G_f)$ and $\mc{A}(G_f)$ from Example \ref{sierpinski}, therefore, $X_{\mc{C}(G_f)}^*=Y_{\mc{A}(G_f)}$ and $X_{\mc{A}(G_f)}^*=Y_{\mc{C}(G_f)}$.
\end{itemize}
\end{exa}

\section{Sierpi\'nski graphs and spaces}\label{sier1}

Let us recall Example \ref{sierpinski}, with a bijection $f\colon \mathbb{N} \to \mathbb{Q}$ we associate the following structures: 
\begin{itemize}\setlength\itemsep{0.1cm}
\item The partial order $\leq_f$ on $\NN$, $n\leq_f k$ iff $n\leq k$ and $f(n)\leq f(k)$. 
\item The \emph{Sierpi\'nski graph} $G_f=G_{(\NN,\leq_f)}=(\NN,E_f)$, i.e. the comparability graph of $(\NN,\leq_f)$, $\{n,k\}\in E_f$ iff $n\ne k$ and either $n\leq_f k$ or $k\leq_f n$.
\item The family $\mc{C}_f=\mc{C}(G_f)$ and the generated \emph{Sierpi\'nski space} $X_f=X_{\mc{C}_f}$  with norm $\|\bullet\|_f=\|\bullet\|_{\mc{C}_f}$ and canonical basis $(e_{f,n})=(e_{\mc{C}_f,n})$. Notice that $\mc{A}_f=\mc{A}(G_f)=\mc{C}_{-f}$. Also, let $Y_f=Y_{\mc{C}_f}$.
\end{itemize}

Denote by $\mrm{Bij}(\NN,\mbb{Q})$ the set of all bijections from $\NN$ to $\mbb{Q}$. First of all, we characterize the cases when different bijections lead to the same graphs and spaces, also, the isomorphisms between these graphs: 
 
\begin{prop}\label{nonisom} If $f,h\in \mrm{Bij}(\NN,\mbb{Q})$ and $\pi$ is a permutation of $\NN$, then the following are equivalent: 
\begin{itemize}\setlength\itemsep{0.1cm}
\item[(a)] $\pi:G_f\to G_h$ is an isomorphism; 
\item[(b)] $(e_{f,n})\stackrel{i}{\sim}(e_{h,\pi(n)})$; 
\item[(c)] $\pi=\mrm{id}$ and there is an order isomorphism $e$ of $\mbb{Q}$ such that $f=e\circ h$.
\end{itemize}
\end{prop}
\begin{proof} Obviously, (c)$\to$(a)$\leftrightarrow$(b). To prove (a)$\to$(c) first we show that $\pi$ preserves the usual ordering of $\NN$ hence $\pi$ must be the identity. Fix $n_0 < n_1 \in \mathbb{N}$ and suppose for the contradiction that $\pi(n_1)<\pi(n_0)$. We will assume that $f(n_0) < f(n_1)$, the case when $f(n_0)>f(n_1)$ is analogous. Pick an $n\in (n_1,\infty)$ such that 
\[ f(n_0) < f(n) < f(n_1)\,\text{ and }\,\pi(n_0)<\pi(n).\]  
Now, $f(n_0)<f(n)$ implies $\{n_0,n\}\in E_f$, hence $\{\pi(n_0),\pi(n)\}\in E_h$, therefore $\underline{h(\pi(n_0))<h(\pi(n))}$,  
and similarly, $f(n)<f(n_1)$ leads to $\underline{h(\pi(n))<h(\pi(n_1))}$ and $f(n_0)<f(n_1)$ leads to $\underline{h(\pi(n_1)) < h(\pi(n_0))}$,
a contradiction.

\smallskip
Define the bijection $e:\mbb{Q}\to\mbb{Q}$, $e(h(n))=f(n)$. Given $n<k$, we know that $f(n)<f(k)$ iff $\{n,k\}\in E_f=E_h$ iff $h(n)<h(k)$, hence $e$ is order preserving.
\end{proof} 

It follows that, up to isomorphism, there are many Sierpi\'nski graphs, and, up to (permutative) isometric equivalence, there are many Sierpi\'nski spaces. These graphs and spaces form somewhat strange classes, namely, each of their elements is universal:

\begin{lem}\label{copy} If $f,h\in\mrm{Bij}(\NN,\mbb{Q})$ then there is an increasing (with respect to $\NN$) isomorphic embedding of $G_h$ in $G_f$, in other words, $(e_{h,n})$ is isometrically equivalent to a subbasis of $(e_{f,n})$. In particular, $X_f$ contains an complemented isometric copy of $X_h$. 
\end{lem}
\begin{proof}
The desired embedding $G_h\to G_f$, $i\mapsto n_i$ can be constructed by recursion: Suppose we already have $n_1<n_2 < \dots < n_k$ such that $i\mapsto n_i$ is an isomorphism between $G_h\clrest \{1,\dots,k\}$ and $G_f\clrest\{n_1,\dots,n_k\}$. Choose a $q\in \mathbb{Q}\setminus\{f(n):n\leq n_k\}$ such that $q < f(n_i)$ iff $h(k+1) < h(i)$ for every $i=1,\dots,k$ and set $n_{k+1}=f^{-1}(q)$.
\end{proof}

Applying this property of the Sierpi\'nski spaces, we show that working with isometries when distinguishing their bases was crucial (for similar results, see
\cite{Wojtowicz}):

\begin{thm}\label{permisom} If $f,h\in\mrm{Bij}(\NN,\mbb{Q})$, then $(e_{f,n})$ and $(e_{h,n})$ are permutatively equivalent. In particular, all Sierpi\'nski spaces are isomorphic.
\end{thm}
\begin{proof} Applying Lemma \ref{copy}, fix injections $\varphi,\psi:\mathbb{N} \to \mathbb{N}$ such that $(e_{f,n})\stackrel{i}{\sim}(e_{h,\varphi(n)})$ and $(e_{h,n})\stackrel{i}{\sim}(e_{f,\psi(n)})$. By the Banach Decomposition Theorem, there are partitions $\NN=A_1\cup A_2=B_1\cup B_2$ such that $\varphi[A_1] =  B_1$ and $\psi[B_2] = A_2$. Define $\rho: \mathbb{N} \to \mathbb{N}$ by setting $\rho(n)=\varphi(n)$ for $n\in A_1$ and $\rho(n) = \psi^{-1}(n)$ for $n\in A_2$, then $\rho$ is a bijection.	The point is that $\rho\clrest A_i:A_i\to B_i$ is an isomorphism between (the hypergraphs) $\mc{C}_f\clrest A_i$ and $\mc{C}_h\clrest B_i$. It follows that if $x\in c_{00}$ and $\supp(x)\subseteq A_i$, then $\|\sum_n x(n)e_{f,n}\|_f=\|\sum_n x(n)e_{h,\rho(n)}\|_h$, and hence for arbirtrary $x\in c_{00}$, 
\begin{align*} 
\Big\|\sum_nx(n)e_{f,n}\Big\|_f & \leq \Big\|\sum_{n\in A_1}x(n)e_{f,n}\Big\|_f+\Big\|\sum_{n\in A_2}x(n)e_{f,n}\Big\|_f\\
&=\Big\|\sum_{n\in A_1}x(n)e_{h,\rho(n)}\Big\|_h+\Big\|\sum_{n\in A_2}x(n)e_{h,\rho(n)}\Big\|_h\leq 2\Big\|\sum_nx(n)e_{h,\rho(n)}\Big\|_h.
\end{align*}
Applying this argument along $\rho^{-1}$, we obtain $\|\sum_nx(n)e_{h,\rho(n)}\|_h\leq 2\|\sum_nx(n)e_{f,n}\|_f$, and so $(e_{f,n})$ and $(e_{h,\rho(n)})$ are equivalent.
\end{proof}

\begin{cor} If $f\in\mrm{Bij}(\NN,\mbb{Q})$, then $X_f^*$ is isomorphic to $Y_f$. 
\end{cor}
\begin{proof}
Applying Theorem \ref{permisom} with $h=-f$, $X_f\simeq X_{-f}=X_{\mc{A}_f}$, and  $X_{\mc{A}_f}^*=Y_f$.
\end{proof}

The above property is quite exceptional in the light of $X_f\subseteq Y_f\subseteq X^{**}_f$.

\smallskip
A very natural question is which Banach spaces can be embedded in the Sierpi\'nski space. Recall that $X_\mc{C}$, where $\mc{C}$ consists of all finite chains in $2^{<\NN}$, is universal for Banach spaces with
unconditional bases but $X_\mc{A}$, where $\mc{A}$ is the set of all finite anti-chains in $2^{<\NN}$, is not: It possesses the weak Banach-Sacks property (and so it does not contain a subspace without this property, e.g. the Schreier space) and it
does not contain a copy of $c_0(\ell_1)$, see \cite{BangOdell}.

In the next two sections we will show that the Sierpi\'nski space is quite rich, namely, it contains copies of various classical Banach spaces, including all Schreier spaces of finite order and all $\ell_p$ spaces. At the same time, it is not universal for Banach spaces with unconditional bases because it does not contain copies of $c_0(\ell_1)$.  

\begin{lem}\label{blockbasic}
Given a Banach space $X$ with basis $(e_n)$. If $X$ contains a copy of $c_0(\ell_1)$, then there is a block basic sequence\footnote{Here, as the indices are not linearly ordered, block basic means that the supports are finite and pairwise disjoint.} with respect to $(e_n)$ which is equivalent to the canonical basis of $c_0(\ell_1)$.
\end{lem}
\begin{proof}
We can assume that $\|e_n\|=1$, let $b_{n,k}\in X$ be a normalized basic sequence equivalent to the basis of $c_0(\ell_1)$, $n$ along the $c_0$ product, $k$ along the copies of $\ell_1$, $Y=[b_{n,k}]\simeq c_0(\ell_1)$, $Y_n=[b_{n,k}:k\in\NN]\simeq\ell_1$. Denote  $e_n^*\in X^*$ and $b^*_{n,k}\in Y^*$ the bi-orthogonal functionals. For $x\in X$, let $\mrm{supp}_X(x)\subseteq\NN$ be its support with respect to $(e_n)$, and for $x\in Y$, $\mrm{supp}_Y(x)\subseteq\NN\times\NN$ its support with respect to $(b_{n,k})$. 

Fix a sequence $(n_l)$ of natural numbers such that $\{l:n=n_l\}$ is infinite for every $n\in\NN$ and a sequence $(\eps_l)$ of elements of $(0,1)$. By recursion on $l$, we will pick $y_l\in Y$ and $x_l\in X$ such that the following holds: 
\begin{itemize}\setlength\itemsep{0.1cm}
\item[(0)] $y_l\in Y_{n_l}$ for every $l$,
\item[(1)] $(y_l)$ is a seminormalized block basic sequence with respect to $(b_{n,k})$,
\item[(2)] $(x_l)$ is a seminormalized block basic sequence with respect to $(e_n)$,
\item[(3)] $\|y_l-x_l\|<\eps_l$ for every $l$.
\end{itemize}

Assuming we can do so, $(y_l)$ is (permutatively) equivalent to the basis of $c_0(\ell_1)$ (because of (0) and (1)), and if $(\eps_l)$ converges quickly enough to $0$, then $(y_l)$ and $(x_l)$ are equivalent because of the principle of small perturbations (see
\cite[Theorem 1]{Pelczynski58}).

\smallskip
For $l=1$, let $y_1=b_{n_1,1}$ and pick a large enough $N$ such that $x_1=\sum_{i=1}^Ne^*_i(y_1)e_i$ is as desired. Now, assume that we already have $y_1,\dots,y_l$ and $x_1,\dots,x_l$. Let $S=\bigcup_{i=1}^l\mrm{supp}_Y(y_i)\subseteq\NN\times\NN$, $s=\max\bigcup_{i=1}^l\mrm{supp}_X(x_i)$, and 
\[ \wt{Y}=Y_{n_{l+1}}\cap\bigcap_{(n,k)\in S}\ker(b^*_{n,k})\cap\bigcap_{j=1}^{s}\ker(e^*_j).\]
Then $\wt{Y}\subseteq Y_{n_{l+1}}$ is a closed subspace of finite co-dimension. Fix a unit vector $y \in \wt{Y}$, then $\mrm{supp}_Y(y)\subseteq (\{n_{l+1}\}\times \NN)\setminus S$, and so we can pick a projection $y_{l+1}$ of $y$ along a finite $E\subseteq (\{n_{l+1}\}\times \NN)\setminus S$ such that $\|y_{l+1}-y\|<\eps_l/2$. Also, as $\mrm{supp}_X(y)\subseteq \NN\setminus \{1,\dots,s\}$, we can pick a $x_{l+1}$ of finite support $\subseteq\NN\setminus \{1,\dots,s\}$ such that $\|x_{l+1}-y\|<\eps_l/2$.  
\end{proof}

Before the next ingredient, we recall a variant of Pt\'ak's Lemma: 

\begin{lem}\label{Ptak} \emph{(see \cite[Theorem 7.2]{BNF20})} 
Let $\mu:\mc{P}(\NN)\to [0,\infty]$ be a measure such that $\mu(\NN)=\infty$ and $\mu(\{n\})\to 0$, $\mc{F}\in\mrm{FHC}$, and $\eps>0$. If for every finite $E\subseteq\NN$ there is $F\in \mc{F}$ such that $F\subseteq E$ and $\mu(F) \geq \eps\mu(E)$, then $\mc{F}$ is not compact.
\end{lem}

\begin{lem}\label{brutal-copy} 
Let $\mc{F}\in\mrm{FHC}$ and suppose that a normalized block basic sequence $(x_n)$ in $X_\mc{F}$ is $K$-equivalent to the basis of $\ell_1$. Then there is an $A\in \overline{\mathcal{F}}$ such that $\|P_A(x_n)\|_\mathcal{F}>1/(2K)$ for infinitely many $n$.
\end{lem}

\begin{proof} Define the family
\[ \mathcal{G} = \big\{G\in [\mathbb{N}]^{<\infty}: \exists\,F\in \mathcal{F}\;\forall\,n\in G\; \|P_F(x_n)\|_\mc{F} > 1/(2K)\big\}\in\mrm{FHC}. \]

First we show that for every finite $E\subseteq\mathbb{N}$ there is a $G\in\mc{G}\clrest E$ such that 
\[\tag{$\ast$} \frac{1}{2K-1} \sum_{n\in E} \frac{1}{n}\leq \sum_{n\in G} \frac{1}{n}.\] 
As $(1/K)\sum_{n\in E}1/n\leq \|\sum_{n\in E}(1/n)x_n\|_\mc{F}$ and $x_n$ is finitely supported, we can fix an $F\in \mathcal{F}$ such that 
\[ \frac{1}{K} \sum_{n\in E} \frac{1}{n}\leq \Big\|\sum_{n\in E} \frac{1}{n}P_F(x_n)\Big\|_\mc{F}.\] 
Therefore, with $G = \{n \in E: \| P_F(x_n)\|_\mc{F} > 1/(2K)\}\in\mc{G}$,  
\begin{align*} 
\frac{1}{K} \sum_{n\in E} \frac{1}{n} &\leq \Big\| \sum_{n\in G} \frac{1}{n} P_F(x_n)\Big\|_\mc{F} + \Big\| \sum_{n\in E \setminus G} \frac{1}{n} P_F(x_n)\Big\|_\mc{F}\\ 
&\leq \sum_{n\in G} \frac{1}{n} + \frac{1}{2K}\sum_{E\setminus G} \frac{1}{n}=\sum_{n\in G}\frac{1}{n} + \frac{1}{2K}\sum_{n\in E}\frac{1}{n} - \frac{1}{2K}\sum_{n\in G}\frac{1}{n} \end{align*}
and this implies ($\ast$).

\smallskip
Now, applying Lemma \ref{Ptak} for the measure $\mu(E) = \sum_{n\in E} 1/n$ on $\NN$, the family $\mc{G}$, and $\eps=1/(2K-1)$, we conclude that $\mathcal{G}$ cannot be compact and so there is an infinite $B\in \overline{\mathcal{G}}$. For $m\in \NN$ fix an $F_m \in \mathcal{F}$ such that $\| P_{F_m}(x_n)\|_\mc{F}>1/(2K)$ for each $n\in B\cap \{1,\dots,m\}\in\mc{G}$; a subsequence $(F_m)_{m\in H}$ of $(F_m)$ tends to an $A\in\overline{\mc{F}}$. We claim that $\| P_A(x_n)\|_\mc{F} > 1/(2K)$ for every $n\in B$. Fix an $n\in B$ and a $k$ above the support of $x_n$, then $A\cap \{1,\dots,k\}=F_m\cap\{1,\dots,k\}$ for all but finitely many $m\in H$; pick such an $m\geq n$, then 
\[ \big\|P_A(x_n)\big\|_\mc{F}\geq \big\|P_{A\cap\{1,\dots,k\}}(x_n)\big\|_\mc{F}=\big\|P_{F_m\cap\{1,\dots,k\}}(x_n)\big\|_\mc{F}=\big\|P_{F_m}(x_n)\big\|_\mc{F}> \frac{1}{2K}.\qedhere\]
\end{proof}

\begin{thm}\label{c_0-l_1} The Sierpi\'nski space does not contain copies of $c_0(\ell_1)$.
\end{thm}
\begin{proof} Fix a bijection $h\colon \mathbb{N} \to \mathbb{Q}$ and suppose for the contradiction that a normalized basic sequence $(x_{n,k})$ in $X_h$ is $K$-equivalent to the canonical basis of $c_0(\ell_1)$, $n$ along the $c_0$-product, $k$ along the copies of $\ell_1$. Applying Lemma \ref{blockbasic}, we can assume that $(x_{n,k})$ is a block basic sequence with respect to $(e_{h,n})$; and, by thinning out the sequence if necessary, that $\mrm{supp}(x_{n,k})<\mrm{supp}(x_{n,k+1})$ for every $n,k$.\footnote{Given non-empty finite subsets $E$ and $F$ of a linear order $(L,\leq)$, we write $E<F$ if $\max(E)<\min(F)$, also, $E\leq F$ analogously.}  Applying Lemma \ref{brutal-copy}, for every $n$ there is an $A_n\in\overline{\mc{C}}_h$ such that $\|P_{A_n}(x_{n,k})\|_h>1/(2K)$ for infinitely many $k$. Keeping these $k$s only, we can assume that $\| P_{A_n}(x_{n,k}) \|_h > 1/(2K)$ for every $n,k$. Let $F_{n,k}=A_n\cap \mrm{supp}(x_{n,k})$. 

To finish the proof, we will show that for every $m\in\NN$ there are an $I\subseteq\mathbb{N}$ of size $m$ and a sequence $(k_n)_{n\in I}$ such that $\bigcup_{n\in I} F_{n,k_n}\in \mc{C}_h$. This is impossible because then $\|\sum_{n\in I}x_{n,k_n}\|_h\geq \|\sum_{n\in I}P_{F_{n,k_n}}(x_{n,k_n})\|_h>|I|/2K$ contradicting the $K$-equivalence along the $c_0$-product with $m=|I|\geq 2K^2$.

For each $n$ let $r_n = \sup_k \max h[F_{n,k}]\in \mbb{R}\cup\{\infty\}$. Notice that $h[F_{n,k}]<h[F_{n,k+1}]<r_n$ for every $n,k$ because $F_{n,k}<F_{n,k+1}$ and $\bigcup_k F_{n,k}\subseteq A_n\in\overline{\mc{C}}_h$. Given an $m$, there is a one-to-one (but not necessarily increasing) sequence $(n_i)_{i=1}^m$ such that $r_{n_1}\leq r_{n_2}\leq\dots\leq r_{n_m}$. Let $I = \{n_i\colon i=1,\dots,m\}$ and let $k_{n_1} = 1$. Suppose that we have defined $k_{n_i}$ for some $i<m$. If $k$ is large enough, then $F_{n_i,k_{n_i}} < F_{n_{i+1},k}$, and, since $h[F_{n_i,k_i}]<r_{n_i}\leq r_{n_{i+1}}$, $h[F_{n_i,k_{n_i}}] < h[F_{n_{i+1},k}]$ also holds for all but finitely many $k$. Pick such a $k$ to be $k_{n_{i+1}}$, then $\bigcup_{i=1}^m F_{n_i,k_i}\in\mc{C}_h$.
\end{proof}

\begin{que}
Does the Sierpi\'nski space contain a copy of $\ell_1(c_0)$?
Is there a reasonable characterization of those Banach spaces with unconditional basis (or those combinatorial spaces) which embed in the Sierpi\'nski space?
\end{que}

\section{Combinatorial spaces emulated in Sierpi\'nski spaces}\label{emulations}

In this section we study canonical subspaces of the Sierpi\'nski space(s). For example, each of them contains obvious (complemented isometric) copies of $c_0$ and $\ell_1$: Given $f\in\mrm{Bij}(\NN,\mbb{Q})$, there are both infinite chains and infinite anti-chains with respect to $\leq_f$. If $D\in \overline{\mc{C}}_f$ is infinite, then $(e_{f,n})_{n\in D}$ is isometrically equivalent to the basis of $\ell_1$, and similarly, if $B\in\overline{\mc{A}}_f$ is infinite, then $(e_{f,n})_{n\in B}$ is isometrically equivalent to the basis of $c_0$.

\smallskip
Below we will describe a general scheme of producing subspaces of the Sierpi\'nski space. First of all, notice that the definitions of $\leq_f$, $G_f$, $\mc{C}_f$, and $X_f$ make perfect sense when starting with an injection $\theta:\mathbb{N} \to \mathbb{Q}$ instead of a bijection $f$. Moreover, the basis $(e_{\theta,n})$ of the resulting combinatorial space $X_\theta$ is isometrically equivalent to a subbasis of $(e_{f,n})$ in a `proper' Sierpi\'nski space $X_f$: Simply extend $\mbb{N}$ to a larger copy $\NN'$ of the natural numbers by adding intermediate points and extend $\theta$ to a bijection $f$ between $\NN'$ and $\mbb{Q}$. In particular, every Sierpi\'nski space contains a complemented isometric copy of $X_\theta$. We will embed certain combinatorial spaces in these `partial' Sierpi\'nski spaces. To emphasize that the underlying set of the combinatorial space $X_\mc{F}$ we wish to embed in an $X_\theta$ does not play any role (unlike $\NN$ in the defintion of the Sierpi\'nski spaces), we will work with a general countable $\Omega$. 
  
\begin{df}
Let $\mc{F}\in\mrm{FHC}(\Omega)$. We say that a pair $((I_t)_{t\in \Omega},\theta)$ \emph{emulates} $\mc{F}$ if $(I_t)$ is a sequence of pairwise disjoint non-empty intervals on $\NN$, $\theta:\NN\to\mbb{Q}$ is an injection, and, with $x_t=\sum_{n\in I_t}e_{\theta,n}\in X_\theta$, for every $E\in [\Omega]^{<\infty}$
\[ \Big\|\sum_{t\in E}x_t\Big\|_\theta=|E|\;\text{ if and only if }\,E\in\mc{F}.\]
If there is such a pair, we say that $\mc{F}$ \emph{admits an emulation}.
\end{df}

Notice that $\|\sum_{t\in E}x_t\|_\theta=|E|$ for $E=\{t\}$ means that $\|x_t\|_\theta=1$, i.e. $\theta$ is decreasing on $I_t$. In other words, an emulation, for now over an infinite $\Omega$, can be seen as the pair of an enumeration $\Omega=\{t_n:n\in\NN\}$ and a sequence $(R_t)$ of pairwise disjoint non-empty finite subsets of $\mbb{Q}$: The enumeration gives us the order of the intervals, $I_{t_1}<I_{t_2}<\dots$, $|I_t|=|R_t|$, and $\theta$ is the unique function $\bigcup_tI_t\to\mbb{Q}$ which is decreasing on $I_t$ and $\theta[I_t]=R_t$; obviously, $\theta$ on $\NN\setminus \bigcup_tI_t$ does not play any role. It also follows that $I_t$ itself is not important but $|I_t|$ and the order of the intervals; this will be useful later when we will need to move / shift the intervals. In general, 
\[\Big\| \sum_{t\in E} x_t\Big\|_\theta=\max\Big\{|C|:C\in \mc{C}_\theta\,\text{ and }\,C\subseteq\bigcup_{t\in E}I_t\Big\}\] for every finite $E\subseteq\Omega$, where, of course, if $C\cap I_t\ne\0$ then $|C\cap I_t|=1$. In other words, if we label elements of $R_{t_n}$ with $n$, then $\| \sum_{t\in E} x_t\|_\theta$ is the length of the longest increasing sequence in $\bigcup_{t\in E} R_t$ with increasing labels.

Applying these observations, it is easy to check the following: 

\begin{fact}
If $((I_t),\theta)$ emulates $\mc{F}$, then $(x_t)\stackrel{i}{\sim}(e_{\mc{F},t})$; in particular, the Sierpi\'nski space contains an isometric copy of $X_\mc{F}$.
\end{fact}
	
\begin{exa}[emulating $\bigcup\mc{Q}$ and $\mrm{Fh}(\mc{Q})$]\label{c_0} Let $I_n = \{n\}$ for each $n$ and $\theta: \mathbb{N} \to \mathbb{Q}$ be increasing. Then $((I_n), \theta)$ emulates $[\NN]^{<\infty}$, and so $(x_n)$ generates a copy of $\ell_1$. Similarly, if $\theta$ is decreasing, then it emulates $[\NN]^{\leq 1}$. 

Regarding $\bigcup\mc{Q}$ and $\mrm{Fh}(\mc{Q})$ from Example \ref{farah}, we show that both of them admit emulations assuming every $\mc{Q}_n$ does (see also the remark below). Let $\NN=\bigcup_nV_n$ be a decomposition into non-empty finite sets, $\mc{Q}=(\mc{Q}_n)$ be a sequence where $\mc{Q}_n\in\mrm{FHC}(V_n)$, and assume that $((I^n_t)_{t\in V_n},\theta_n)$ emulates $\mc{Q}_n$. 

We can assume that $I^1:=\bigcup_{t\in V_1}I^1_t<I^2:=\bigcup_{t\in V_2}I^2_t<\dots$ and $\bigcup_nI^n=\mbb{N}$. We can also move the ranges of each $\theta_n$, if we shift them to the left such that $\theta_1[I^1]>\theta_2[I^2]>\dots$, then $((I^n_t),\bigcup_n\theta_n)$ emulates $\bigcup\mc{Q}$; if we shift them to the right such that $\theta_1[I^1]<\theta_2[I^2]<\dots$, then $((I^n_t),\bigcup_n\theta_n)$ emulates $\mrm{Fh}(\mc{Q})$.
\end{exa}

\begin{rem}\label{balazs} (Bal\'azs Keszegh) 
In general, not every finite family, that is, $\mc{F}\in\mrm{FHC}(V)$ for some finite $V\ne\0$, admits emulations. Moreover, among those which satisfy $\mc{F}\subseteq [V]^{\leq 2}$, most of them do not admit emulations. This follows from the observation that in this case we can assume that $|I_t|\leq 2$ for every $t\in V$, and hence, with $n=|V|$, the number of emulations is $\leq n^nn^{2n}$ but the number of such, pairwise non-isomorphic families is asymptotically $2^{\binom{n}{2}}/n!$. Probably the simplest concrete example of such an $\mc{F}$ without emulation is $\mc{C}(C_5)$.  
\end{rem}

\begin{exa}[$\ell_1\oplus_0\ell_1$ and $c_0\oplus_1 c_0$]
Let $\Omega=\{0,1\}\times\NN$ and $\mc{F}=[\{0\}\times\NN]^{<\infty}\cup [\{1\}\times\NN]^{<\infty}$, in other words, $\mc{F}$ is the family generating $\ell_1\oplus_0\ell_1$, the product equipped with the max norm. Of course, Sierpi\'nski spaces contain copies of $\ell_1\oplus_0\ell_1$ (because it is isomorphic to $\ell_1$). At the same time, it is easy to show that $\mc{F}$ admits no emulation: Assume on the contrary that $((I_{i,n}),\theta)$ emulates $\mc{F}$ ($i\in\{0,1\}$ and $n\in\NN$). Ordering the intervals consecutively $J_k=I_{i_k,n_k}$, $J_1<J_2<\dots$, we can find $k<m<k'$ such that $i_k=i_{k'}\ne i_m$ and we may assume that $i_k=0$. This means that $\theta[J_m]<\theta[J_k]$ because otherwise $\{(0,n_k),(1,n_m)\}$ would belong to $\mc{F}$, and similarly $\theta[J_{k'}]<\theta[J_m]$, and so $\theta[J_{k'}]<\theta[J_k]$ but this is impossible because $\{(0,n_k),(0,n_{k'})\}\in\mc{F}$.  

The same holds for the family $\mc{F}=[\{0,1\}\times\NN]^{\leq 1}\cup\{\{(0,n),(1,m)\}:n,m\in\NN\}$ generating $c_0\oplus_1c_0$, the product equipped with the sum norm: Fix $J_k<J_m<J_{k'}<J_{m'}$ such that $i_k=i_{k'}=0$ and $i_m=i_{m'}=1$. Then $\theta[J_{m'}]<\theta[J_m]$ because $\{(1,n_m),(1,n_{m'})\}\notin\mc{F}$, $\min\theta[J_m]<\max\theta[J_{k'}]$ because $\{(1,n_m),(0,n_{k'})\}\in\mc{F}$, and $\theta[J_{k'}]<\theta[J_k]$ because $\{(0,n_k),(0,n_{k'})\}\notin\mc{F}$, therefore $\theta[J_{m'}]<\theta[J_k]$, a contradiction becaue $\{(0,n_k),(1,n_{m'})\}\in\mc{F}$. 
\end{exa}

\begin{que}
Does the Sierpi\'nski space contain isometric copies of $\ell_1\oplus_0\ell_1$ or $c_0\oplus_1 c_0$?
\end{que}

In the rest of this section, we will discuss the Schreier families and their possible emulations (for more details and results on these families and the generated spaces, see e.g. \cite{Alspach-Argyros}, \cite{Castillo-Gonzales},
\cite{Gasparis-Leung}). We will need the following (see \cite{CauseySzlenk} for $\mbf{S}$ and $\mbf{D}$): The \emph{Schreier operation} $\mbf{S}:\mrm{FHC}\to\mrm{FHC}$, 
\[ \mbf{S}(\mathcal{F}) =\{\0\}\cup  \bigg\{\bigcup_{k=1}^mF_k:m\in\NN,F_k \in \mathcal{F}\setminus\{\0\},\text{ and }\,\{m\}\leq F_1 < \dots < F_m\bigg\},\]
the \emph{diagonal union operation}  $\mbf{D}:\mrm{FHC}^\NN\to\mrm{FHC}$,
\[ \mbf{D}\big((\mc{F}_k)_k\big)=\bigcup_{k=1}^\infty\mc{F}_k\clrest[k,\infty)=\{\0\}\cup \big\{F\colon F\in \mathcal{F}_k\setminus\{\0\}, k\leq \min F\big\},\]
and a combination $\mbf{D}^*:\mrm{FHC}^\NN\to\mrm{FHC}$ of $\mbf{S}$ and $\mbf{D}$,
\[ \mbf{D}^*\big((\mc{F}_k)_k\big) = \{\0\}\cup\bigg\{\bigcup_{k=1}^mF_k:m\in\NN,F_k \in \mathcal{F}_k\setminus\{\0\},\text{ and }\,\{m\}\leq F_m < \dots < F_1\bigg\}.\]
In order to define the Schreier families, we fix a \emph{ladder system}
\[ \boldsymbol\xi=\big((\xi^\ga_k)_{k=1}^\infty:\ga\in \mrm{Lim}(\om_1)\big)\]
where $\mrm{Lim}(\om_1)=\{$limit ordinals below $\om_1\}$ and $\xi^\ga_1<\xi^\ga_2<\dots<\ga$ is unbounded in $\ga$ for each such $\ga$. Given a ladder system $\boldsymbol\xi$, the associated \emph{Schreier} and \emph{Schreier$^*$} sequences, $\mc{S}_\al(\boldsymbol\xi)$ and $\mc{S}^*_\al(\boldsymbol\xi)$ are constructed by recursion on $\al<\om_1$: Let $\mc{S}_0(\boldsymbol\xi)=\mc{S}^*_0(\boldsymbol\xi)=[\NN]^{\leq 1}$, in the successor step let 
\[ \mc{S}_{\al+1}(\boldsymbol\xi)=\mbf{S}\big(\mc{S}_\al(\boldsymbol\xi)\big)\,\text{ and }\,\mc{S}^*_{\al+1}(\boldsymbol\xi)=\mbf{S}\big(\mc{S}^*_\al(\boldsymbol\xi)\big),\] 
and for a limit $\ga$ let 
\[ \mc{S}_\ga(\boldsymbol\xi)= \mbf{D}\big(\big(\mathcal{S}_{\xi^\ga_k}(\boldsymbol\xi)\big)_k\big)\,\text{ and }\,\mc{S}^*_\ga(\boldsymbol\xi)= \mbf{D}^*\big(\big(\mathcal{S}^*_{\xi^\ga_k}(\boldsymbol\xi)\big)_k\big).\] 

It is common to assume a bit more about the generating ladder system, namely, we call $\boldsymbol\xi$ \emph{standard} if 
\[\tag{acc} \mc{S}_{\xi^\ga_1}(\boldsymbol\xi)\subseteq\mc{S}_{\xi^\ga_2}(\boldsymbol\xi)\subseteq\dots\,\text{ is increasing for every }\,\ga.\] 
It is easy to construct standard ladders by recursion applying the following: An easy induction shows that given an arbitrary ladder system $\boldsymbol\xi$ and $\al<\be<\om_1$, 
\[ \text{there is an }\,m\in\NN\,\text{ such that }\,\mc{S}_\al(\boldsymbol\xi)\clrest [m,\infty)\subseteq\mc{S}_\be(\boldsymbol\xi),\]
and hence $\mc{S}_\al(\boldsymbol\xi)\subseteq \mc{S}_{\be+m}(\boldsymbol\xi)$. Now, if $\boldsymbol\xi$ is standard, then so is $\boldsymbol\xi^+=((\xi^\ga_k+1)_{k=1}^\infty:\ga\in\mrm{Lim}(\om_1))$ and it is straightforward to check that 
\[ \mc{S}_\al(\boldsymbol\xi)\subseteq\mc{S}_\al^*(\boldsymbol\xi)\subseteq \mc{S}_\al(\boldsymbol\xi^+)\,\text{ for every }\,\al.\]

From now on, if we simply write $\mc{S}_\al$ and $\mc{S}^*_\al$, we refer to any families of the form $\mc{S}_\al(\boldsymbol\xi)$ and $\mc{S}^*_\al(\boldsymbol\xi)$ for some standard $\boldsymbol\xi$. 

\smallskip
We will use the following preservation theorem: 

\begin{thm}\label{lem:Schreier} 
Suppose that $\mc{F},\mc{F}_k\in\mrm{FHC}$, $k\in\NN$. Then the following holds:
\begin{itemize}\setlength\itemsep{0.1cm}
\item[(1)]
If $\mc{F}$ has an emulation $((I_n),\theta)$ satisfying $I_1<I_2<\dots$, then so does the family $\mbf{S}(\mc{F})$.
\item[(2)] If each $\mc{F}_k$ has an emulation $((I^k_n),\theta_k)$ such that $I^k_1<I^k_2<\dots$, then so does the family 
$\mbf{D}^*((\mathcal{F}_k)_k)$.	
\end{itemize}
\end{thm}

\begin{proof} (1): 
We may assume that $\theta[\mathbb{N}] \subseteq (0,1)$. Let $(J_n)$ be the interval partition of $\NN$ such that $J_n = I^1_n \cup \dots \cup I^n_n$, where $| I^k_n | = |I_n|$ for each $k$ and $I^1_n < I^2_n < \dots < I^n_n$. Fix order preserving bijections $g^k_n\colon I^k_n \to I_n$ and define $\theta^* \colon \mathbb{N}\to \mathbb{Q}$: 
\[ \text{For }\,j\in I^k_n\subseteq J_n\,\text{ let }\,\theta^*(j)=\theta\big(g^k_n(j)\big)-k.\] In other words, $\theta^*$ acts like $\theta$ on each $I^k_n\sim I_n$ but moves the images to the left of each other as $k$ increases. It follows that $\theta^*$ is decreasing on $J_n$.

We claim that $((J_n),\theta^*)$ emulates $\mbf{S}(\mathcal{F})$, that is, with $x_n=\sum_{j\in J_n}x_{\theta^*,j}\in X_{\theta^*}$, \[ \Big\|\sum_{n\in E}x_n\Big\|_{\theta^*}=|E|\,\text{ iff }\,E\in \mbf{S}(\mc{F}).\]

\smallskip
First, let $E= F_1 \cup \dots \cup F_m\in \mbf{S}(\mc{F})$ where $F_k\in\mc{F}\setminus\{\0\}$ and $\{m\}\leq F_1<F_2<\dots<F_m$. We need to find an increasing sequence $(j_n)_{n\in E}$ such that $j_n\in J_n$ and $(\theta^*(j_n))_{n\in E}$ is also increasing.
As $\mathcal{F}$ is emulated by $((I_n), \theta)$, for each $k$ there is an increasing sequence $(i_n)_{n \in F_k}$ such that $i_n\in I_n$ and $(\theta(i_n))_{n\in F_k}$ is also increasing. Take the copies of these $i_n$ from $I^{m-k+1}_n$ (we know that $k\leq m\leq n$), that is, let $j_n\in I^{m-k+1}_n\subseteq J_n$ such that $g^{m-k+1}_n(j_n)=i_n$. Then $(j_n)_{n\in F_k}$ is increasing, just like
\[ \theta^*(j_n)=\theta\big(g^{m-k+1}_n(j_n)\big)-(m-k+1)=\theta(i_n)-m+k-1\;(n\in F_k).\]
The concatenation $(j_n)_{n\in E}$ of the sequences $(j_n)_{n\in F_k}$, $k=1,\dots,m$ is increasing; we show that its image $(\theta^*(j_n))_{n\in E}$ is also increasing: If $n\in
F_k$ and $n'\in F_{k+1}$, then $j_n\in I^{m-k+1}_n\subseteq J_n$ and $j_{n'}\in I^{m-k}_{n'}\subseteq J_{n'}$, hence
\[ \theta^*(j_n)=\theta(i_n)-m+k-1<-m+k<\theta(i_{n'})-m+k=\theta^*(j_{n'}).\]

\smallskip
Now suppose that $\|\sum_{n\in E} x_n \|_\theta = |E|$ holds for some finite $E\subseteq\NN$ witnessed by $(j_n)_{n\in E}$, and let $m=\min(E)$. First of all, notice that if $n\in E\setminus\{m\}$ and $j_n\in I^k_n$, then $k\leq m$ because otherwise $\theta^*(j_n)=\theta(g^k_n(j_n))-k<-m$ and, at the same time, with some $k'\leq m$, $\theta^*(j_n)>\theta^*(j_m)=\theta(g^{k'}_m(j_m))-k'\geq -m$. For each $k\leq m$ define 
\[ F_k=\big\{n\in E:j_n\in I^{m-k+1}_n\big\}.\] Then $F_k\in\mc{F}$ because for $n\in F_k$ both $i_n=g^{m-k+1}_n(j_n)$  and $\theta(i_n)=\theta^*(j_n)+m-k+1$ are increasing. Some of these sets may be empty but if $n\in F_k$ and $n'\in F_{k'}$ for some $k<k'\leq m$, then 
\[ \theta^*(j_n)=\theta(i_n)-m+k-1<-m+k'-1<\theta(i_{n'})-m+k'-1=\theta^*(j_{n'})\]
and hence $F_k<F_{k'}$. It follows that $E\in \mbf{S}(\mc{F})$.

\smallskip
(2): We can
assume that $\theta_k[\mathbb{N}] \subseteq (0,1)$ for each $k$. Also, we can rearrange the intervals $I^k_n$ by shifting them along $\NN$ such that
\[ I^1_n<\dots<I^n_n<I^1_{n+1}<\dots\,\text{ and }\,J_n:=I^1_n\cup\dots\cup I^n_n\,\text{ is an interval}.\]
Define $\theta^*\colon \mathbb{N}\to \mathbb{Q}$ as follows: 
\[ \text{For }\,j\in I^k_n\subseteq J_n\,\text{ let }\,\theta^*(j)=\theta_k(j)-k.\] 
The obvious modification of the proof of (1) shows that $((J_n),\theta^*)$ emulates $\mbf{D}^*((\mathcal{F}_k)_k)$: We pick $j_n\in I^k_n\subseteq J_n$, instead of $j_n\in I^{m-k+1}_n$, and concatenate the sequences $(j_n)_{n\in F_k}$ in the reverse order, $(j_n)_{n\in F_m}, (j_n)_{n\in F_{m-1}},\dots,(j_n)_{n\in F_1}$. Similarly, in the second part of the proof, we work with $I^k_n$ instead of $I^{m-k+1}_n$.
\end{proof}

\begin{rem} When applying Theorem \ref{lem:Schreier} to obtain an emulation of $\mc{S}_1=\mc{S}$ from the canonical emulation of $\mc{S}_0=[\NN]^{\leq 1}$, $|J_n| = n$ for each $n\in\NN$ and $\theta^*[\mathbb{N}]$ will be the union of infinitely
many decreasing sequences. It follows that the closure of $\theta^*[\mathbb{N}]$ is of Cantor-Bendixson rank $2$. Similarly, if $((J_n),\theta^*)$ is the canonical emulation of $\mc{S}_2$ constructed in the theorem above, then the closure of $\theta^*[\mathbb{N}]$ is of Cantor-Bendixson rank $3$, and so on. 
\end{rem}

Since $\mc{S}_0=\mc{S}^*_0=[\NN]^{\leq 1}$ has an emulation as required above, we obtain the following by induction:

\begin{cor}\label{Schreier}
For every $n\in\NN$, $\mc{S}_n$ admits an emulation, and so the Sierpi\'nski space contains isometric copies of each Schreier space $X_{\mc{S}_n}$ of finite order. Furthermore, for every $\al<\om_1$, $\mc{S}^*_\al$ admits an emulation, and hence the Sierpi\'nski space contains isometric copies of each Schreier$^*$ space $X_{\mc{S}^*_\al}$. 
\end{cor}

It is natural to ask the following :

\begin{que}
Does the Schreier families $\mc{S}_\alpha$ admit an emulation for each $\alpha$? Or, at least, can $\mc{S}_\alpha$ be embedded into the Sierpi\'nski space isomorphically?
\end{que}

On the other hand, it seems that for most purposes, when $\mc{S}_\alpha$ is used in the literature, $\mc{S}^*_\alpha$ would serve equally good. Also, it is not clear how the (original) Schreier spaces depend on the ladder system used in their definition. If they do not (with respect to equivalence), then one can show that $\mc{S}^*_\alpha$ is equivalent to $\mc{S}_\alpha$ for each $\alpha$. If they do, then perhaps $\mc{S}^*_\alpha$ deserves the name 'Schreier family of rank $\alpha$' no less than $\mc{S}_\alpha$.

\section{Embedding $\ell_p$'s in the Sierpi\'nski space}\label{sier2}
 
Let $(X,\|\bullet\|)$ be a sequence Banach lattice, that is, $c_{00}\subseteq X\subseteq\mbb{R}^\NN$, the unit vectors $(e_n)$ are normalized in $\|\bullet\|$ and for any $x\in X$ and $y\in\R^\N$, if $|y(n)|\leq |x(n)|$ for each $n\in\NN$, then $y\in X$ with $\|y\|\leq \|x\|$. In follows that $(e_n)$ is a 1-unconditional basic sequence in $X$, in particular  for any $x=\sum_na_nx_n\in X$ and $A\subseteq\NN$ also $x_A=\sum_{n\in A}a_nx_n\in X$ with $\|x_A\|_X\leq \|x\|_X$. Let $S(X)$ be the unit sphere in $X$ and $\widetilde B(X)=B(X)\cap c_{00}$.

Consider a partial order in $\tilde B(X)$ defined as follows: $x\dotleq y$, for $x,y\in \tilde B(X)$, if for any $z\in \tilde B(X)$ with $\min \supp z>\max\{\max \supp x, \max \supp y\}$, and $x+z\in \tilde B(X)$ also $y+z\in \tilde B(X)$. 

Consider the following properties of the basic sequence $(e_n)\subseteq X$.
\begin{enumerate}\setlength\itemsep{0.1cm}
	\item[($\alpha$)] for any $x,y\in \tilde B(X)$ with a finite support, either $x\dotleq y$ or $y\dotleq x$,  
	\item[($\beta$)] for any $x,y\in \tilde B(X)$,  $n\in\N$ with $\max\{\max\supp x,\max\supp y\}<n$, and scalars $a,b\in [0,1]$: if $y\dotleq x+ay_n\in \tilde B(X)$ with $x+(a+b)y_n\in \tilde B(X)$ then $\tilde B(X)\ni y+by_n\dotleq x+(a+b)y_n$. 
\end{enumerate}


\begin{prop}    \label{modular}
Let $X$ be a sequence Banach lattice admitting a function $I: X\to [0,\infty)$ with the following properties:
\begin{enumerate}\setlength\itemsep{0.1cm}
    \item $I(x)\leq 1\iff \|x\|\leq 1$ for any $x\in X$, and $I(0)=0$,
    \item $I$ is additive on disjointly supported vectors, 
	\item $I_n: [0,\infty)\ni t\mapsto I(te_n)\in [0,\infty)$ is continuous and superadditive (for example, convex) for any $n\in\N$.
 \end{enumerate}
 Then $x\dotleq y\iff I(x)\leq I(y)$ for any $x,y\in B(X)$. Moreover, the unit vector sequence $(e_n)\subseteq X$ satisfies $(\alpha)$ and $(\beta)$.  
\end{prop}
\begin{proof}    
	Fix $x,y\in \tilde B(X)$. 

	Assume first that $I(x)\leq I(y)$ and take $z\in X$ with \[ \min\supp z>\max\{\max\supp x,\max\supp y\} \] and $y+z\in B(X)$. Then by (1) and (2) we have 
\[1\geq I(y+z)=I(y)+I(z)\geq I(x)+I(z)=I(x+z)\] 
hence $x+z\in B(X)$ by (1). For the reverse implication, towards contradiction, assume $x\dotleq y$ and $I(x)>I(y)$. For $n>\max\{\max\supp x, \max\supp y\}$, by (1) and (3), take $s>0$ with $I(se_n)=1-I(y)$. By (2) and (1), $y+se_n\in \tilde B(X)$,
whereas $x+se_n\not\in\tilde B(X)$, which contradicts $x\dotleq y$. 

The property $(\alpha)$ follows immediately. For $(\beta)$ take $x,y,n,a,b$ as required and proceed, using (2), (3) and the above characterization of $\dotleq$ in terms of $I$,  as follows:
\begin{align*}    
I(x+(a+b)e_n)&=I(x)+I((a+b)e_n)\geq I(x)+I(ae_n)+I(be_n)\\
&=I(x+ae_n)+I(be_n)\geq I(y)+I(be_n)=I(y+be_n)
\end{align*}
\end{proof}    

We now recall the class of Banach spaces that admit a modular with the properties listed in Proposition  \ref{modular}, based on \cite{maligranda,musielak}. 

Let $\Phi=(\phi_i)$ be a sequence of Orlicz functions, i.e. each $\phi_i:[0,\infty)\to [0,\infty]$ is convex, left-continuous, not vanishing on $\phi_i^{-1}[0,\infty)$, with $\phi_i(0)=0$ and $\lim\limits_{t\to\infty}\phi_i(t)=\infty$. We can also
assume that $\phi_i(1)=1$ for all $i\in\N$. A Musielak-Orlicz space (or a generalized Orlicz space) induced by $\Phi$ is the space
\[\ell_\Phi:=\Big\{x=(a_i)\in \R^\N: I_\Phi\big(\tfrac{x}{\rho}\big)<\infty \text{ for some }\rho>0\Big\},\]
where \[I_\Phi(x):=\sum_{i=1}^\infty \phi_i(|a_i|), \ x=(a_i)\in \R^\N, \]
with the norm $\|x\|_\Phi=\inf\{\rho>0: I_\phi(\tfrac{x}{\rho})<\infty\}$. 
It follows that the conjugate function $\psi_i$ of $\phi_i$, $i\in\N$, in the sense of Young, i.e. given by $\psi_i(s)=\sup\{st-\phi_i(t): t\geq 0\}$, $s\geq 0$, 
is also an Orlicz function. 

We consider also the following condition, called the $\Delta_2$ condition: for some $K>1$ and $(h_i)\in \ell_1\cap (0,\infty)^\N$
\[\phi_i(2t)\leq K\phi_i(t)+h_i. \ \ \mbox{ for }i\in\N\]

\begin{thm}\label{Orlicz spaces} With the above notation, $\ell_\Phi$ is reflexive iff
	both $\Phi=(\phi_i)$ and $\Psi=(\psi_i)_i$ satisfy $\Delta_2$. Also, every reflexive $\ell_\Phi$ satisfies the following properties:
\begin{enumerate}\setlength\itemsep{0.1cm}
    \item the unit vectors $(e_i)$ form a basis both in $\ell_\Phi$ and $\ell_\Psi$,
	\item $\ell_\Phi^*$ is isometrically isomorphic to $\ell_\Psi$, with $\{e_i\colon i\in \mathbb{N}\} \subseteq \ell_\Psi$ isometrically equivalent to the sequence of bi-orthogonal functionals to $\{e_i\colon i\in \mathbb{N}\} \subseteq \ell_\Phi$,
    \item $I_\Psi(x)\leq 1\iff \|x\|\leq 1$ for every $x\in\ell_\Psi$. 
\end{enumerate}
\end{thm}

\begin{rem} \hfill 
\begin{enumerate}\setlength\itemsep{0.1cm}
    \item Any Banach space $X$ with a normalized 1-unconditional basis $(x_n)$ is isometrically isomorphic with some Banach lattice $\tilde X$ with the unit vector basis $(e_n)$. Indeed, let $\tilde X=\{(a_n)\in\R^\NN: \sum_n a_nx_n\in X\}$ endowed with the norm $\|(a_n)\|=\|\sum_na_nx_n\|_X$ for any $(a_n)\in \tilde X$. Then trivially the map $X\ni \sum_n a_nx_n\mapsto (a_n)\in\tilde X$ is an isometric isomorphism. Therefore we can speak about properties $(\alpha)$ and $(\beta)$ of the basis $(x_n)$. 
    \item Let $X$ be a Banach space with a normalized 1-unconditional basis $(x_n)$. Then the sequence $(x_n^*)$ of functionals bi-orthogonal  to $(x_n)$ is also a normalized 1-unconditional basic sequence, forming a weak* basis of $X^*$, that is any $f\in X^*$ is of the form $f=w^*\lim_n f(x_n)x^*_n$. Let $Z=\{(f(e_n)): f\in X^*\}$ endowed with the norm $\|(f(e_n))\|=\|f\|_{X^*}$, $f\in X^*$. Then $Z$ is a Banach lattice isometrically isomorphic to $X^*$, with the unit vector basis in $Z$ isometrically equivalent to $(x_n^*)$ in $X^*$. In particular, we can speak about properties $(\alpha)$ and $(\beta)$ of $(x_n^*)$. 
    \item Any $\ell_p$, $1\leq p<\infty$ is a Musielak-Orlicz space, with Orlicz functions $\phi_i(t)=t^p$, $t\in [0,\infty)$, $i\in\NN$. 
    \item  Let $X$ be a reflexive Musielak-Orlicz space, considered with the canonical unit vector basis. Then, by Theorem \ref{Orlicz spaces} and Proposition \ref{modular}, the sequence of bi-orthogonal functionals to the basis 
    satisfies $(\alpha)$ and $(\beta)$.        
\end{enumerate}    
\end{rem}

\begin{thm} \label{Banach spaces in Sierpinski}
    Let $Y$ be a Banach space with a  normalized 1-unconditional basis $(y_n)$ so that the sequence of bi-orthogonal functionals $(y_n^*)\subseteq Y$ satisfies $(\alpha)$ and $(\beta)$. Then the Sierpi\'nski space contains an isomorphic copy of $Y$. 

 In particular, the Sierpi\'nski space contains an isomorphic copy of  any reflexive  Musielak-Orlicz sequence space. It follows that the Sierpi\'nski space contains an isomorphic copy of any $\ell_p$, $1\leq p<\infty$ and $c_0$. 
\end{thm}
\begin{proof}    	Let $K$ be the set of all functionals in $B(Y^*)$ with non-negative coefficients and finite support, note that $K$ contains $(y^*_n)$ and is invariant under restrictions to subsets of $\N$, i.e.
\begin{equation}\label{K restrictions}
	\text{ for every }g=\sum_{n=1}^N a_ny^*_n\in K \text{ and } E\subseteq \{1,\dots,N\} \text{ we have }\sum_{n\in E}a_ny^*_n\in K.
\end{equation}

We write $(e_j)$ for the unit vector basis of $c_{00}$.

We shall choose inductively an injection $f:\N\to\Q\cap (0,\infty)$ and a normalized block sequence $(x_n)$ in  $X_f$, each $x_n$ of the following form
\begin{equation}\label{xx_n}
	x_n=\tfrac{1}{4^nM_n}\sum_{m=M_n}^0\sum_{j\in J_n^{m}}e_j,
\end{equation}
where $(J_n^{m})_{m=M_n}^0$ are successive finite intervals in $\N$. 

With every finite $F\in \mathcal{C}_f$ (i.e. an element of the Sierpi\'nski family generated by $f$) we assign a functional  $g_F\in Y^*$ of the form
\[g_F=\sum_n\Big(\sum_{j\in F}e^*_j(x_n)\Big)y_n^*. \]
In particular, for any $F\in\mathcal{C}_f$ and scalars $(b_n)$, 
\begin{equation}\label{X_phi to Y}
	g_F\Big(\sum_nb_ny_n\Big)=\Big(\sum_{j\in F}e^*_j\Big)\Big(\sum_nb_nx_n\Big).
\end{equation}
We aim at the following properties:
\begin{equation}\label{K to F_phi}
	g\in K\Longrightarrow \|g-g_F\|<\tfrac{1}{2} \text{ for some }F\in\mathcal{C}_f, 
\end{equation}
\begin{equation}\label{F_phi to K}
	F\in \mathcal{C}_f \Longrightarrow \|g-g_F\|<\tfrac{1}{2} \text{ for some }g\in K.
\end{equation}
By \eqref{X_phi to Y} and 1-unconditionality of the basis $(e_j)$ of $X_f$, \eqref{K to F_phi} yields the equivalence of $(x_n)$ in $X_f$ and $(y_n)$ in $Y$, in particular $(x_n)$ is seminormalized. 

In order to achieve \eqref{K to F_phi}, we proceed inductively, at each $n$-th step building a vector $x_n$ as in \eqref{xx_n} with $\min \supp x_n>\max\supp x_{n-1}$,  extending $f$ to an injection $D_n:=\supp x_1\cup\dots\cup \supp x_n\to\Q\cap
[0,\infty)$, and defining auxiliary sets $(F_t)_{t\in f[\supp x_n]}$ of integers satisfying the following properties (with $K_n=K\cap \spa\{y_1^*,\dots,y^*_n\}$ and $\varepsilon_n=\tfrac{1}{2}-\tfrac{1}{4^n}$): 
\begin{enumerate}\setlength\itemsep{0.1cm}   
	\item[a)] For every $t\in f[D_n]$ we have $F_t\in\mathcal{C}_f$ with $\max f[F_t]=t$ and $g_{F_t}\in K$. 
\item[b)] For every $t,s\in f[D_n]$ with $s\leq t$ we have $g_{F_s}\dotleq g_{F_t}$.
\item[c)] For every $F\in\mathcal{C}_{f}(\subseteq 2^{D_n})$, there is $G\in\mathcal{C}_f$ with
	$\|g_F-g_G\|<\varepsilon_n$, $\max f[G]\leq \max f[F]$, $g_G\in K$,  $g_G\dotleq g_{F_t}$ for $t:=\max f[F]$.
\item[d)] Every $g\in K_n$  satisfies $\|g-g_{F_t}\|<\varepsilon_n$ for some  $t\in f[D_n]$.
	
\end{enumerate}
Note that (c), resp. (d), implies immediately \eqref{F_phi to K}, resp. \eqref{K to F_phi}, by 1-unconditionality of $(y_n^*)$. 

We describe now the inductive construction. 
Let $M_1=0$ and pick an interval $J_1^{0}=\{1,2,3,4\}\subseteq\N$. Put $f :J_1^{0}\ni j\to j\in\Q$ and define $x_1$ according to \eqref{xx_n}. Let $F_t=J^{0}_1\cap [1,t]$ for any $t\in f[D_1]=\{1,2,3,4\}$. Then all the desired conditions are satisfied. 

Now, fix $n>1$ and assume that we have built $x_1,\dots, x_{n-1}$, and $f: D_{n-1}=\supp x_1\cup\dots\cup\supp x_{n-1}\to \Q$   satisfying conditions (1)-(4).

Enumerate $f[D_{n-1}]$ as $t_1<\dots<t_M$. Let $M_n=M$ and for each $m=1,\dots,M$ 
pick $g_m:=g_{F_{t_m}}\in K_{n-1}$  by (1) of the inductive assumption. Put also $g_0=0$.

In the sequel we will drop the index $n$ for the sake of simplicity. At the $n$-th step of the construction, we perform further inductive construction described below. We shall choose inductively on $m=M,\dots,0$ (starting from $m=M$) the following objects:

\begin{itemize}\setlength\itemsep{0.1cm}
    \item scalars $(a_l^{m})_{l=0}^{L^{m+1}}\subseteq A:=\{\tfrac{k}{N}: k=0,1,\dots,N\}$, with $N=4^nM$,  
    \item successive intervals $(I_l^{m})_{l=1}^{L^{m+1}}$ whose union $J^{m}:=\bigcup_{l=1}^{L^{m+1}}I_l^{m}$ is also an interval with $\min J^{m}>\max J^{m+1}$, with $\min J^{M}>\max \supp x_{n-1}$, 
    \item an extension of $f$ to an injection: $D_{n-1}\cup J^{M}\cup\dots\cup J^{m}\to\Q\cap [0,\infty)$,
    \item ($L^{m}$ and) an increasing sequence  $(t_l^{m})_{l=0}^{L^{m}}\subseteq [t_m,\infty] \cap \mathbb{Q}$ with $t_0^{m}=t_m$, $t_{L^{m}}^{m}=\infty$, 
    \item sets  of integers $F_t$, $t\in \{t_l^{m}: l=1,\dots, L^{m}\}$,    
\end{itemize} 
as follows (conditions listed according to the choice order).
\begin{enumerate}[label=(\roman*)]\setlength\itemsep{0.1cm}
    \item $L^{M+1}=1$ (needed to start the induction for $m=M$ below in (ii)),
    \item $a_0^{m}=0$,  $a_l^{m}$ is maximal in $A$ with $g_m+a_l^{m}y^*_n\dotleq g_{F_t}$, where $t=t_l^{m+1}$ for each $l=1,\dots, L^{m+1}-1$, and $a_{L^{m+1}}^{m}$ is maximal in $A$ with $g_m+a_{L^{m+1}}^{m}y^*_n\in K$, 
    \item $|I_l^{m}|=(a_l^{m}-a_{l-1}^{m})N\in\{0,1,\dots,N\}$ for each $l=1,\dots, L^{m+1}$,  (thus $|J^{m}|=Na_{L^{m}}^{m}$),
    \item $f$ is increasing on $J^{m}$,
    \item $f[I_l^{m}]\subseteq (t_{l-1}^{m+1}, t_l^{m+1})$ for all $l=1,\dots, L^{m+1}$, 
	\item $t_0^{m}=t_m$ and $(t_l^{m})_{l=1}^{(L^{m})}$ is an increasing enumeration of the set $(t_l^{m+1})_{l=0}^{(L^{m+1})}\cup f[J^{m}]$ (which also defines $L^{m}$),
	\item $F_t=F_{t_m}\cup\{\min J^{m}, \min J^{m}+1,\dots,f^{-1}(t)\}$ for every $t\in f[J^{m}]$. 
\end{enumerate}
Note that conditions (ii), (iii) and (v)  in presence of the condition $(\alpha)$ yield 
\begin{enumerate}\setlength\itemsep{0.1cm}
	\item[(viii)] $g_{F_t}\dotleq g_{F_s}\iff t<s$ for any $s\in f[D_{n-1}\cup J^{M}\cup\dots\cup J^{m+1}]$, $t\in f[J^{m}]$.
\end{enumerate}
We also check at each $m$-th step (1), (2) and (3) with $\supp x_1\cup\dots\cup\supp x_{n-1}\cup J^{M}\cup\dots\cup J^{m}$ replacing $D_n$ and $\varepsilon^{m}=\varepsilon_{n-1}+\tfrac{M-m}{4^nM}$ replacing $\varepsilon_n$. 

We justify now the inductive construction (for $m=M,\dots,0$). 

The construction for $m=M$ is clear by (i). (1) and (2) follow immediately by the inductive assumption on $n-1$ and definition, in particular (ii) and (viii). For (3) pick any $F\in\mathcal{C}_f$, $F\subseteq D_{n-1}\cup J^{M}$, then $F=F'\cup J'$
with $F'\subseteq D_{n-1}$ and $J'\subseteq J^{M}$, let $t=\max f[J']$. By the inductive assumption (3) for $n-1$, there is $G'\in\mathcal{C}_f|_{D_{n-1}}$ with $\|g_{F'}-g_{G'}\|<\varepsilon_{n-1}$,  $g_{G'}\in K_{n-1}$, $\max f[G']\leq \max
f[F']=t$ and $g_{G'}\dotleq g_{F_t}$. As $t\leq t_M$, by (2) we have $g_{F_t}\dotleq g_M$. As by (ii) $g_M+a_{L^{M}}^{M}y_n^*\in K$, also $g_M+\tfrac{|J'|}{N}y_n^*\in K$ by 1-unconditionality of $(y_n^*)$. For $G=G'\cup J'\in\mathcal{C}_f$, by
$g_{G'}\dotleq g_M$, also $g_G=g_{G'}+g_{J'}=g_{G'}+\tfrac{|J'|}{N}y_n^*\in K$. Moreover, $g_G\dotleq g_M+g_{J'}\dotleq g_t$, as $J'\subseteq \{\min J^{M},\dots,f^{-1}(t)\}$. To finish the proof of (3) note that $\|g_G-g_F\|<\varepsilon_{n-1}$.

Fix $m<M$ and assume we have performed the construction for $m+1,m+2,\dots,M$. 

The condition (ii)  well defines scalars $(a_l^{m})_l$ by (3) for $m+1$, the rest of the definition is clear. (1) follows directly by the definition of $a_{L^{m}}^{m}$ in (ii), (vii) and (iv).  

For (2) fix $s<t$ in $f[D_{n-1}\cup J^{M}\cup\dots\cup J^{m}]$ with at least one of them in $f[J^{m}]$ (otherwise we use the inductive assumption for $m+1$). If both $s,t\in f[J^{m}]$, then (vii) yields $g_{F_s}\dotleq g_{F_t}$. If $s\not\in
f[J^{m}]\ni t$, then $g_{F_s}\dotleq g_{F_t}$ goes by (viii). If $s\in f[J^{m}]\not\ni t$, then $g_{F_t}\cancel{\dotleq} g_{F_s}$ by (viii), thus $g_{F_s}\dotleq g_{F_t}$ by condition $(\alpha)$.    

For (3) take $F\in\mathcal{C}_f$, $F\subseteq D_{n-1}\cup J^{M}\cup\dots\cup J^{m}$ with $t=\max f[F]\in f[J^{m}]$. Then, as $\max J^{m+1}<\min J^{m}$, $F=F'\cup J'$ with $F'\subseteq D_{n-1}\cup J^{M}\cup \dots\cup J^{m+1}$ and $J'\subseteq J^{m}$
with $t=\max f[J']=t\geq\min f[J']>\max f[F']=:s$. By the inductive assumption for $m+1$, there is $G'\in\mathcal{C}_f$, $G'\subseteq D_{n-1}\cup J^{M}\cup\dots\cup J^{m+1}$ with $\|g_{G'}-g_{F'}\|<\varepsilon^{m+1}$, $\max f[G']\leq \max f[F']=s$,
$K\ni g_{G'}\dotleq g_{F_s}$. Write $F_t=F_{t_m}\cup\{\min J^{m},\dots,f^{-1}(t)\}$ by (vii). 

If $s\leq t_m$, then by (2) $g_{F_s}\dotleq g_m$, so $g_{F_s}+g_{J'}\dotleq g_m+g_{J'}\dotleq g_{F_t}$ as $J'\subseteq \{\min J^{m}, \dots, f^{-1}(t)\}$, so $g_{G'} + g_{J'} \dotleq g_{F_t}$. 

If $s>t_m$, then $s=t^{m+1}_l$ for some $1\leq l\leq L^{m+1}$. 
As $\min f[J']>s$ and $J'\subseteq J^{m}$, $g_m+(a_l^{m}+\tfrac{|J'|}{N})y_n^*\in K$. In particular,
$a_l^{m}<a_{L^{m}}^{m}$, and condition $(\alpha)$ and maximality of $a^{m}_l$ yields $g:=g_{F_s}\dotleq g_m+(a_l^{m}+\tfrac{1}{N})y^*_n=: f$. Apply  condition $(\beta)$ to $g$, $f$, $a=a_l^{m}+\tfrac{1}{N}$ and $b=\tfrac{|J'|-1}{N}$ to get 
$K\ni g_{F_s}+\tfrac{|J'|-1}{N}y^*_n\dotleq g_m+(a_l^{m}+\tfrac{|J'|}{N})\dotleq g_{F_t}$. 
Finally, for $G=G'\cup J'\setminus \{\max J'\}\in \mathcal{C}_f$ we get $\|g_F-g_G\|<\varepsilon^{m+1}+\tfrac{1}{N}\leq \varepsilon^{m}$ and $K\ni g_G\dotleq g_{F_t}$, which proves (3) and thus ends the $m$-th step of the inductive construction
(satisfying at each step the suitable variants of (1), (2) and (3) - with $D_{n-1}\cup J^{M}\cup\dots\cup J^{m}$ replacing $D_n$ and $\varepsilon^{m}$ replacing $\varepsilon_n$.

Once we complete the inductive procedure inside the $n$-th step, we use $(J^{m})_{m=M}^0$ to define $x_n$ by \eqref{xx_n}, whereas $f$ on $\supp x_n=\bigcup_{m=M}^0J^{m}$ is already chosen by (iv) and (v). Note that (1), (2) and (3) are satisfied for
$D_n=D_{n-1}\cup J^{M}\cup\dots J^{0}$ and $\varepsilon_n\geq\varepsilon^{0}$. Therefore, we need only to check (4). Take $g\in K_n$ and write it as $g=g'+ay^*_n$ for some $g'\in K_{n-1}$ and $a\neq 0$. By the inductive assumption (4) for $n-1$ and
by the definition of $t_1<\dots<t_M$, $\|g'-g_m\|<\varepsilon_{n-1}$ for some $1\leq m\leq M$. Pick $k\in\N$ with $\tfrac{k-1}{4^n}\leq a<\tfrac{k}{4^n}$, note that $k\leq 4^na_{L^{m}}^{m}+1=|J^{m}|+1$ by (ii) and (iii). Thus for $t=f (\min
J^{(m)}+k-2)$ we obtain $\|g-g_{F_t}\|<\varepsilon_{n-1}+\tfrac{1}{4^n}\leq \varepsilon_n$, which proves  (4) and ends the $n$-th step of the general construction. 
\end{proof}

\section{Remarks on the extreme points in graph-generated spaces}\label{remext}

Given a graph $G=(V,E)$ let $X_G=X_{\mc{C}(G)}$, $Y_G=Y_{\mc{C}(G)}$, and $\|\bullet\|_G=\|\bullet\|_{\mc{C}(G)}$. In Remark \ref{extYF} it was pointed out that for a perfect $G$ \begin{align*}
\mrm{Ext}(B(Y_G))&=W\big(\max\big(\overline{\mc{A}(G)}\big)\big)\\
&=\big\{y\in\{-1,0,1\}^V:\mathrm{supp}(y)\,\text{ is a
maximal anti-clique}\big\}.
\end{align*} 

It is natural to ask what we can say about $\mrm{Ext}(B(Y_G))$  when $G$ is not perfect. This was a subject of Sebastian Jachimek's doctoral thesis \cite{SebTh}. We will recall some of these results, with his kind permission, skipping some details of
the proofs. We will call elements of $W(\max(\overline{\mc{A}(G)}))$ \emph{terminal points} (of $G$). First of all, note that in the general case $\mrm{Ext}(B(Y_G))$ contains all extreme points of $B(X_G)$ as well as all terminal points of $G$. By the strong perfect graph theorem (Theorem \ref{strong-perfect}), if $G$ is not perfect, then it contains either an odd hole or an odd antihole (of length $\geq 5$). Now, it is not hard to see that if $G = (V,E)$ itself is a $(2n+1)$-hole, then $x\colon V \to \mathbb{R}$, $x(v) = 1/2$ for each $v\in V$, is an extreme point of $B(Y_G)$. We show that if $G$ contains an odd hole $D$, and $x:D\to\mbb{R}$, $x(v)=1/2$ for $v\in D$, then it can be extended to an extreme point of $B(Y_G)$:

\begin{thm}\label{extendingextreme} If $G = (V,E)$ contains an odd hole $D$, then there is an extreme point of $B(Y_G)$ whose value is $1/2$ on each $v\in D$ (so it is not terminal) and $\mrm{ran}(x)\subseteq\{0,1/2,1\}$.
\end{thm}

\begin{proof} We assume that $V = \{v_0, v_1, \dots\}$ is infinite. We will find the desired extreme point $x$ in a recursive way, setting more and more values of $x$. We may assume, without loss of generality, that $G$ is connected. Indeed, suppose
	that a connected component $C_0$ contains an odd hole and we are able to find $x_0$ supported on $C_0$. For every other component $C$ we can fix a terminal point $x_C$ supported by this component, and then $x = x_0 + \sum_C x_C\in\mrm{Ext}(B(Y_G))$.

\smallskip
Let $D_0=D$ and define $x(v)=1/2$ for each $v\in D_0$. At the $(n+1)$st step let 
\[ k = \min\big\{j\in \mathbb{N}: v_j \in V \setminus D_n\,\text{ and }\,v_j\,\text{ is connected to a vertex in }\,D_n\big\}. \]
Since $G$ is infinite and connected, $k$ is well-defined and, with $D_{n+1} = D_n \cup \{v_k\}$, $\bigcup_n
D_n = V$. At the $(n+1)$st step we define $x(v_k)$ as follows: 
\begin{itemize}\setlength\itemsep{0.1cm}
	\item if there is $v\in D_n$ connected to $v_k$ such that $x(v)=1$, then put $x(v_k)=0$,
	\item if there are $v \ne v'\in D_n$ such that $x(v)=x(v')=1/2$, and $v,v',v_k$ are pairwise connected, then put $x(v_k)=0$,
	\item if all $v\in D_n$ connected to $v_k$ are such that $x(v)=0$, then put $x(v_k)=1$,
	\item otherwise, put $x(v_k)=1/2$.
\end{itemize}

We will sketch the proof that $x\in\mrm{Ext}(B(Y_G))$. Let $V_i = x^{-1}(i)$ for $i\in \{0,1/2,1\}$. Suppose that $z$ is such that $\| x \pm z \|_G \leq 1$. Notice that $\mathrm{supp}(z) \subseteq V_{1/2}$: Indeed, clearly $\mathrm{supp}(z) \cap V_1 = \emptyset$, and if $v\in V_0$, then it is connected either to a $w\in V_1$ or it forms a triangle with some $w, w'\in V_{1/2}$. In any case, $z(v)=0$ follows.

It is easy to see that $V_{1/2}$ is connected. We know that $z\clrest D_0\equiv 0$ because $x\clrest D_0$ is an extreme point of $B(Y_{G\upharpoonright D_0})$. Now, fix a $w\in V_{1/2}\setminus D_0$ with $d(w,D_0)=n>0$ where $d$ is the usual distance function on $G\clrest V_{1/2}$. Then there is a chain $w_0,w_1,\dots,w_n=w$ in $V_{1/2}$ such that $d(w_i,D_0)=i$ and it follows that $z(w_n)=-z(w_{n-1})=z(w_{n-2})=\dots=\pm z(w_0)=0$.
\end{proof}

What can we say about odd antiholes?

\begin{thm}\label{circ} Suppose $G=(V,E)$ is an $(2n+1)$-antihole and that $|x(v)|=1/n$ for each $v\in V$. Then $x\in\mrm{Ext}(B(Y_G))$. 
\end{thm}
\begin{proof} Enumerate $V=\{v_1,\dots,v_{2n+1}\}$ along the associated cycle and notice that \[ (1,0,1,0,\dots,1,0,1,0,0) \] 
is the characteristic function of a maximal clique. In fact, all maximal cliques are of size $n$ and their characteristic functions are shifts of the one above, enumerate them as $C_1,\dots,C_{2n+1}$. Let $A$ be the $(2n+1)\times(2n+1)$ matrix such that its $i$th row is the characteristic function of $C_i$. One can show (see \cite[Proposition 4.3.9]{SebTh}) that $\det(A)\ne 0$. \footnote{A matrix of this form, that is, when the rows list all possible shifts of a fixed non-zero vector, is called \emph{circulant}. The fact that these matrices  are invertible follows from the fact that its eigenvalues are combinations of square $n$-roots of identity where the coefficients are given by a discrete Fourier transform of the inducing vector.}

\smallskip
Now fix an $x$ such that $|x(v)|=1/n$ for each $v$, we show that $x$ is an extreme point of $B(Y_G)$. We can assume that $x(v)=1/n$ for each $v$. Let $z\in Y_G$ be such that $\| x\pm z\|_G\leq 1$. Then, with a small enough $\eps>0$, for each maximal clique $C$, we have
\[ 1\pm \sum_{v\in C}\eps z(v)=\sum_{v\in C}\big|1/n\pm \eps z(v)\big|=\sum_{v\in C}\big|x(v)\pm \eps z(v)\big|\leq\|x+\eps z\|_G \leq 1,\] 
and so $\sum_{v\in C}z(v)=0$, i.e. $Az=0$, hence $z=0$.
\end{proof}

Perhaps the extreme points of the above form can be extended to the whole graph containing an odd antihole, like in Theorem \ref{extendingextreme}, but the construction may be much more involved.

\smallskip
Finally, we present an additional application of the idea used in the last proof: 

\begin{cor} For each $q\in \mathbb{Q}\cap [-1,1]$ there is a graph $G$ and an $x\in\mrm{Ext}(B(Y_G))$ such that $x(v)=q$ for some $v$.
\end{cor}
\begin{proof} We can assume that $q=i/n\in (0,1)$ where $i,n\in\NN$ and $n\geq 2$. Let $V = D \cup \{w\}$, where $D$ is a $(2n+1)$-antihole. Fix a non-maximal clique $C$ in $D$ of size $n-i$, connect $w$ to all vertices in $C$, and define $x$ on the resulting graph $G$ by $x(v) = 1/n$ for $v\in D$ and $x(w) = i/n$. We claim
that $x\in\mrm{Ext}(B(Y_G))$. Consider the matrix $A$ associated to $D$ as in the proof of Theorem \ref{circ} and extend it to a $(2n+2)\times (2n+2)$-matrix $B$, 
\[
	B= \begin{bmatrix}
		A & 0 \\
		y & 1 \\
\end{bmatrix}\]
where $y$ is a characteristic function of $C$. Then $\det(B)=\det(A)\ne 0$ and if $\| x\pm z\|_G\leq 1$ then, by the same argument as above, $Bz = 0$, hence $z = 0$.
\end{proof}

The above corollary indicates that, in the general case, the extreme points of $B(Y_G)$ may be much more complicated than the terminal points of $G$.

\bibliographystyle{siam}
\bibliography{bib-norm}
\end{document}